\documentclass[11pt,twoside]{article}

\usepackage{mathrsfs,amsfonts,amsmath,amssymb}
\usepackage{latexsym}
\newtheorem{satz}{Theorem}
\newtheorem{proposition}[satz]{Proposition}
\newtheorem{theorem}[satz]{Theorem}
\newtheorem{lemma}[satz]{Lemma}
\newtheorem{definition}[satz]{Definition}
\newtheorem{corollary}[satz]{Corollary}
\newtheorem{remark}[satz]{Remark}
\newtheorem{example}[satz]{Example}

\def\N{\mathbb {N}}
\def\Z{\mathbb {Z}}
\def\E{\mathsf {E}}
\def\R{\mathbb {R}}
\def\F{\mathbb {F}}

\def\C{{\cal C}}

\def\G{\Gamma}

\def\d{\delta}
\def\D{\Delta}

\def\({\big (}
\def\){\big )}

\def\c{\circ}

\def\ls{\leqslant}
\def\gs{\geqslant}

\def\_phi{\varphi}
\def\eps{\varepsilon}

\def\Gr{{\mathbf G}}
\def\FF{\widehat}
\def\ov{\overline}
\def\Spec{{\rm Spec\,}}

\def\ge{\gs}
\def\le{\ls}

\def\T{\mathsf{T}}
\def\Sym{\mathsf {Sym}}
\def\C{{\mathbb C}}

\author{Shkredov I.D.}
\title{ An application of the sum--product phenomenon to sets having no solutions of several  linear equations
\footnote{
This work was supported by grant
Russian Scientific Foundation RSF 14--11--00433.}
}
\date{}
\begin{document}
\maketitle


\begin{abstract}
We prove that for an arbitrary
$\kappa \le \frac{1}{3}$
any subset of $\F_p$ avoiding $t$ linear equations with three variables has size less than $O(p/t^\kappa)$.
We also find  several applications to problems about so--called non--averaging sets, number of collinear triples and mixed energies.
\end{abstract}

\section{Introduction}

Let $p$ be a prime number, $\F_p$ be the finite field and $A\subseteq \F_p$ be a set.
Consider a linear equation
\begin{equation}\label{f:eq_intr}
    c_1 x_1 + \dots + c_k x_k = b \,,
\end{equation}
where $k\ge 3$ and $c_1,\dots, c_k \neq 0$.
We say that our set $A$ {\it avoids equation} (\ref{f:eq_intr}) if there are no tuples $(x_1, \dots, x_k) \in A^k$ satisfying  (\ref{f:eq_intr}).
Sets avoiding linear equations is a
well--known
subject of Additive Combinatorics and Number Theory, see, e.g., classical papers \cite{Ruzsa-I,Ruzsa-II} about this question.
It is known that if $b=0$ and $c_1+\dots+c_k=0$, then $|A|=o(p)$ as $p\to \infty$ but
in the other cases
one can easily construct a set of positive density avoiding
(\ref{f:eq_intr}).
In this paper we have to deal with the case $k=3$ but instead of one equation we consider several, say,  $t$ of them.
Such problems are considered in articles  \cite{Ruzsa-II}, \cite{CS}, for example.
 For us the basic question is the following: is it true that $|A| = o_t (p)$ as   $t\to \infty$ (and $p\to \infty$ of course)?
Notice  that we do not require $b=0$ or $c_1+c_2+c_3=0$.
It turns out that the answer is positive and the problem is connected with the sum--product phenomenon, see, e.g., \cite{Tao_Vu_book}.
Let us formulate a special but an useful case of the main result of this paper
(our general Theorem \ref{t:main} is contained in section \ref{sec:proof} below).

\begin{theorem}
    Let $A \subseteq \F_p$ be a set, $|A| \gg p^{39/47}$.
    Suppose that $A$ avoids $t$ equations of the form
\begin{equation}\label{cond:main_intr}
    x_1 + a_j x_2 + b_j x_3 = b_j \,,
\end{equation}
    where all $a_j$, $b_j$ are nonzero and
    each point $(a_i,b_i)$ has either a unique abscissa or ordinate or its ratio.
    Then for any $\kappa < \frac{3}{20}$ one has
\begin{equation}\label{f:main_intr1}
    |A| = O\left( \frac{p}{t^\kappa} \right) \,.
\end{equation}
    In another direction,
    there is a set $A\subseteq \F_p$ avoiding $t$ linear equations of form (\ref{cond:main_intr})
    such that
\begin{equation}\label{f:main_intr2}
    |A| \gg \frac{p}{t^{1/2}} \,.
\end{equation}
\label{t:main_intr}
\end{theorem}

Actually, we prove that $\kappa$ in Theorem \ref{t:main_intr} can be doubled in many cases, see section \ref{sec:proof}, so
the
power of $t$ is between
$0.3$
and $0.5$ for such wide class of equations.
The author thinks that
$0.3$
can be improved slightly but he does not believe that  this constant
can be replaced by something strictly greater than $1/3$,
at least it requires  some new ideas,
would imply  a considerable progress in the area and
seems unattainable at the moment
(see Example \ref{ex:Sarkozy} in section \ref{sec:examples} and discussion after Remark \ref{r:expectation}).

The method of the proof is based on precise incidences results  from \cite{misha} and some applications of these results from  \cite{AMRS},
\cite{RRS}, \cite{RSS}.
Usually, theorems of such a sort have to deal with small subsets of $\F_p$.
Considering a dual set, that is, the {\it spectrum} of a set
or, in other words,
the set of large exponential sums,  see section \ref{sec:preliminaries},
we show that these results are applicable sometimes for
{\it large} subsets of $\F_p$, exactly  as in (\ref{f:main_intr1}), (\ref{f:main_intr2}).
In particular, we
prove
the following fact, which is interesting in its own right:
the spectrum always has small multiplicative energy, see Theorem \ref{t:energy*spec} below.

The simplest example of system (\ref{cond:main_intr}) can be obtained if one consider a multiplicative subgroup
$\Gamma \subseteq \F_p \setminus \{ 0\}$ and take just one linear equation $\gamma = \alpha s_1 + \beta s_2$, where $\alpha, \beta \neq 0$ are fixed, $\gamma \in \Gamma$ and $s_1,s_2$
belong to $\Gamma$.
Then  this equation generates $|\Gamma|^2$ another equations $x=\alpha \gamma' y + \beta \gamma'' z$, where $\gamma', \gamma'' \in \Gamma$  and thus
can be studied by the methods of
our
paper.
Another
nontrivial example is given by so--called {\it collinear triples} of the Cartesian product $A\times A$ of a set $A$.
It is easy to see that three points $(a_1,a_2), (b_1,b_2), (c_1,c_2) \in A\times A$ are collinear if
$\frac{b_1-a_1}{c_1-a_1} =  \frac{b_2-a_2}{c_2-a_2} := \lambda$.
Thus any $\lambda$ generates a linear equation $b+(\lambda-1)a - \lambda c= 0$, $a,b,c \in A$ and hence the number of collinear triples is connected with a system of linear equations of type (\ref{cond:main_intr}),
for more details, see section \ref{sec:applications}.
Further
applications can be found in this section.

The paper is organized as follows.
In sections \ref{sec:definitions}, \ref{sec:preliminaries} we give a list of definitions and results, which will be used further in the text.
In section \ref{sec:examples} we consider some examples of families of sets avoiding several equations and prove  lower bound (\ref{f:main_intr2}).
Section \ref{sec:energy} is devoted to the spectrum of a set.
Here we prove in particular, that the spectrum has small multiplicative energy and contains a large subset with even smaller multiplicative energy.
In the next section
we obtain our main Theorem \ref{t:main} which implies Theorem \ref{t:main_intr}.
Finally, section \ref{sec:applications} contains further applications of the main result.

The author is grateful to Sergey Yekhanin, Tomasz Schoen
and Anh Vinh Le
for useful discussions.

\section{Definitions}
\label{sec:definitions}

Let $p$ be a prime number, $\F_p$ be the finite field and denote by $\F_p^*$ the set $\F_p^* = \F_p \setminus \{ 0 \}$.
The field $\F_p$ is the main subject of our paper but let us consider a slightly general context which we will use sometimes.

Let $\Gr$ be an abelian group.
If $\Gr$ is finite, then denote by $N$ the cardinality of $\Gr$.
It is well--known~\cite{Rudin_book} that the dual group $\FF{\Gr}$ is isomorphic to $\Gr$ in this case.
Let $f$ be a function from $\Gr$ to $\mathbb{C}.$  We denote the Fourier transform of $f$ by~$\FF{f},$
\begin{equation}\label{F:Fourier}
  \FF{f}(\xi) =  \sum_{x \in \Gr} f(x) e( -\xi \cdot x) \,,
\end{equation}
where $e(x) = e^{2\pi i x}$
and $\xi$ is a homomorphism from $\FF{\Gr}$ to $\R/\Z$ acting as $\xi : x \to \xi \cdot x$.
We rely on the following basic identities
\begin{equation}\label{F_Par}
    \sum_{x\in \Gr} |f(x)|^2
        =
            \frac{1}{N} \sum_{\xi \in \FF{\Gr}} \big|\widehat{f} (\xi)\big|^2 \,,
\end{equation}
\begin{equation}\label{svertka}
    \sum_{y\in \Gr} \Big|\sum_{x\in \Gr} f(x) g(y-x) \Big|^2
        = \frac{1}{N} \sum_{\xi \in \FF{\Gr}} \big|\widehat{f} (\xi)\big|^2 \big|\widehat{g} (\xi)\big|^2 \,,
\end{equation}
and
\begin{equation}\label{f:inverse}
    f(x) = \frac{1}{N} \sum_{\xi \in \FF{\Gr}} \FF{f}(\xi) e(\xi \cdot x) \,.
\end{equation}
If
$$
    (f*g) (x) := \sum_{y\in \Gr} f(y) g(x-y) \quad \mbox{ and } \quad
        (f\circ g) (x) := \sum_{y\in \Gr} f(y) g(y+x)
        \,,
$$
 then
\begin{equation}\label{f:F_svertka}
    \FF{f*g} = \FF{f} \FF{g} \quad \mbox{ and } \quad \FF{f \circ g} = \FF{f^c} \FF{g} = \ov{\FF{\ov{f}}} \FF{g} \,,
\end{equation}
where for a function $f:\Gr \to \mathbb{C}$ we put $f^c (x):= f(-x)$.
 Clearly,  $(f*g) (x) = (g*f) (x)$ and $(f\c g)(x) = (g \c f) (-x)$, $x\in \Gr$.
 The $k$--fold convolution, $k\in \N$  we denote by $*_k$,
 so $*_k := *(*_{k-1})$.
    In the same way we use multiplicative convolution of two functions $f,g : \F_p \to \C$ which we denote as
$$
    (f\otimes g) (x) := \sum_{y\in \F^*_p} f(y) g(x y^{-1}) \,.
$$
Write for any function $f: \Gr \to \C$
$$
    \| f \|'_\infty := \max_{x\neq 0} |f(x)| \,.
$$

We use in our paper  the same letter to denote a set
$S\subseteq \Gr$ and its characteristic function $S:\Gr\rightarrow \{0,1\}.$
Write $\E^+(A,B)$ for the {\it additive energy} of two sets $A,B \subseteq \Gr$
(see, e.g., \cite{Tao_Vu_book}), that is,
$$
    \E^+(A,B) = |\{ a_1+b_1 = a_2+b_2 ~:~ a_1,a_2 \in A,\, b_1,b_2 \in B \}| \,.
$$
If $A=B$ we simply write $\E^+(A)$ instead of $\E^+(A,A).$
In the same way one can define the {\it multiplicative energy} of two sets $A,B \subseteq \F_p$ as
$$
    \E^\times (A,B) = |\{ a_1 b_1 = a_2 b_2 ~:~ a_1,a_2 \in A,\, b_1,b_2 \in B \}| \,.
$$
    Sometimes we write $\E(A,B)$ if we do not specialise the energy.
Further clearly,
\begin{equation}\label{f:energy_convolution}
    \E^+(A,B) = \sum_x (A*B) (x)^2 = \sum_x (A \circ B) (x)^2 = \sum_x (A \circ A) (x) (B \circ B) (x)
    \,.
\end{equation}
and by (\ref{svertka}),
\begin{equation}\label{f:energy_Fourier}
    \E^+(A,B) = \frac{1}{N} \sum_{\xi} |\FF{A} (\xi)|^2 |\FF{B} (\xi)|^2 \,.
\end{equation}
Also put
$$
    \E^{+}_* (A) = \frac{1}{N} \sum_{\xi \neq 0} |\FF{A} (\xi)|^4  = \E^{+} (A) - \frac{|A|^4}{N} \,.
$$
Let
$$
   \T^+_k (A) := | \{ a_1 + \dots + a_k = a'_1 + \dots + a'_k  ~:~ a_1, \dots, a_k, a'_1,\dots,a'_k \in A \} |
    =
        \frac{1}{N} \sum_{\xi} |\FF{A} (\xi)|^{2k}
        \,.
$$
Also let
$$
    \sigma^+_k (A) := (A*_k A)(0)=| \{ a_1 + \dots + a_k = 0 ~:~ a_1, \dots, a_k \in A \} | \,.
$$
Notice that for a symmetric set $A$, that is, $A=-A$ one has $\sigma_2
(A) = |A|$ and $\sigma^+_{2k} (A) = \T^+_k (A)$.
Having a set $P\subseteq A-A$ we write $\sigma^+_P (A) := \sum_{x\in P} (A\c A) (x)$
and
$\E^+_P (A) := \sum_{x\in P} (A\c A)^2 (x)$.

\bigskip

Given two sets $Q,R\subseteq \Gr$ and a real number $t\ge 1$, we define
$$
    \Sym^+_{t} (Q,R) := \{ x ~:~ |Q\cap (x-R)| \ge t \}
$$
and similar $\Sym^\times_{t} (Q,R)$.

\bigskip

For a positive integer $n,$ we set $[n]=\{1,\ldots,n\}.$
All logarithms are to base $2.$ Signs $\ll$ and $\gg$ are the usual Vinogradov's symbols, that is, $a\ll b$ iff $a=O(b)$.
     We will write $a \lesssim b$ or $b \gtrsim a$ if $a = O(b \cdot \log^c |A|)$, where $A$ is a fixed set and $c>0$ is an absolute constant.
    Notation $a\sim b$ means $a \lesssim b$ and, simultaneously,  $b \lesssim a$.


\section{Examples of sets avoiding several linear equations}
\label{sec:examples}

First of all, let us recall the definitions.
Let $\mathcal{E}$ be a finite family of equations of the form
\begin{equation}\label{f:E_form}
    a_j x + b_j y + c_j z = d_j \,,
\end{equation}
where all $a_j, b_j, c_j$ are nonzero and such that any two triples $(a_j,b_j,c_j)$, $(a'_j,b'_j,c'_j)$ corresponding some equations from $\mathcal{E}$ are not proportional.
In other words, we consider triples  $(a_j,b_j,c_j)$ from $\F_p^* \times \F_p^* \times \F_p^*$ up to
an equivalence relation  $\sim$, namely,
$(a,b,c) \sim (a',b',c')$ iff for some nonzero $\lambda$ the following holds $a' = \lambda a$, $b' = \lambda b$, and $c' = \lambda c$.
Thus the family $\mathcal{E}$ corresponds to a subset of two--dimensional projective plane.
We denote this set as $S(\mathcal{E})$.
We write $|\mathcal{E}|$ for  the cardinality of $S(\mathcal{E})$.
Also notice that
we do not require
$d_j =0$ or $a_j+b_j+c_j=0$
and hence  so--called non--affine equations (see \cite{Ruzsa-I,Ruzsa-II}) are considered by us as well.
We say that a set $A\subseteq \F_p$ is {\it avoiding family $\mathcal{E}$} if there is no $j\in [|\mathcal{E}|]$
and $x,y,z\in A$ such that $a_j x + b_j y + c_j z = d_j$.
In other words, the set $A$ does not satisfy {\it all} equations from $\mathcal{E}$.
Sometimes a little bit more general setting is required.
Let  $A_1,A_2,A_3 \subseteq \F_p$ be three sets.
We say that the triple $(A_1,A_2,A_3)$ {\it avoids family $\mathcal{E}$} if
for any $j\in [|\mathcal{E}|]$
and all $x \in A_1$, $y \in A_2$, $z\in A_3$ we have $a_j x + b_j y + c_j z \neq d_j$.

Of course the size of a set $A$ avoiding equations (\ref{f:E_form}) depends on the geometry of the set
$\mathcal{E}$ or, equivalently, on the set $S(\mathcal{E})$.
We consider
several
rather rough  characteristics  of the set $\mathcal{E}$ and study them.

\begin{definition}
By $\mathcal{T} (\mathcal{E})$ denote the size of the maximal subset in the intersection of $\mathcal{E}$ with one of three planes
$\{ x=1\}, \{y=1\}, \{ z=1\}$
with the property that all non--fixed coordinates in the intersection  are different.
\label{def:T}
\end{definition}

Thus $\mathcal{T} (\mathcal{E}) \le |\mathcal{E}|$ and the bound is attained if, say,
$S(\mathcal{E}) = \{1\} \times \{ (e,e) ~:~ e\in [|\mathcal{E}|] \}$.
Now
let us
obtain a lower bound for the quantity $\mathcal{T} (\mathcal{E})$.

\begin{lemma}
    We have $\mathcal{T} (\mathcal{E}) \ge |\mathcal{E}|^{1/2}$.
\label{l:TE_T}
\end{lemma}
\begin{proof}
Put $s=|S(\mathcal{E})| = |\mathcal{E}|$, $t=\mathcal{T} (\mathcal{E})$.
    Take a maximal subset $J$ of $[|\mathcal{E}|]$
    such that  $(x_j ,y_j ,1) \in \mathcal{E}$ and all elements $x_j$ as well as all elements $y_j$ are different.
    Clearly, $t \ge |J|$.
    Put $R = \{ (x_j,y_j)\}_{j\in J}$ and $S = \{ (x_j,y_j)\}_{j\in [s]}$.
    By the maximality of $R$ we see that any point of $S$ has either the same abscissa or the same ordinate with a point from $R$.
    Thus one can  split $S\setminus R$ into two sets $S_1$, $S_2$ and $R$ into sets $R_1,R_2$ such that any point from $S_1$, $S_2$ shares  common abscissa or ordinate (or both) with some point from $R_1$, $R_2$, respectively.
    Let us split points from $S\setminus R$ having common abscissa and ordinate with some points from $R$ in an arbitrary way.
    Let $r_1 = |R_1|$, $r_2 = |R_2|$, $s_1=|S_1|$, $s_2 = |S_2|$.
    Then, clearly, $r_1+r_2 = |J|$ and $s_1+s_2+r_1+r_2 = s$.
    Suppose that $r_1,r_2 > 0$.
    By average arguments there is some point $x_0$ such that the set
    $\{ (x_0, y,1) ~:~ (x_0,y) \in S_1\}$
    has size at least $s_1/r_1$.
    Similarly, there is some point $y_0$ such that the set
    $\{ (x, y_0,1) ~:~ (x,y_0) \in S_2\}$
    has size at least $s_2/r_2$.
    Without losing  of generality suppose that $\frac{s_2}{r_2} \le \frac{s_1}{r_1}$
    By a well--known property  of the median, we have
    $$
        \frac{s_2}{r_2} \le \frac{s_1+s_2}{r_1+r_2} = \frac{s}{|J|} -1 \le \frac{s_1}{r_1} \,.
    $$
    Hence there is a set $Q$ of the form  $Q = \{ (x_0,q,1) \in S \}$ of size $|Q| \ge s/|J|-1+1 = s/|J|$
    (we add in $Q$ a point from $R_1$).
    If $r_1$ or $r_2$ vanishes then it is easy to see that the existence of such $Q$ follows similarly and even simpler.
    Finally,
    the points $(x_0,q,1)$  from $Q$ are equivalent to $|Q|$ points of the form $(x_0 q^{-1}, 1, q^{-1})$, having different coordinates
    in the plane $\{ y=1\}$.
    Thus, $t\ge s/|J|$.
    Obviously, $\max_{|J|} \{ |J|, s/|J|\} \ge s^{1/2}$ and hence $t\ge s^{1/2}$.
    This completes the proof.
$\hfill\Box$
\end{proof}

\begin{remark}
    Let $\G$ be a subgroup of $\F_p^*$ and $S(\mathcal{E}) = (\G \times \G \times \G) / \sim$.
    Then in view of $\G/\G=\G$, we have  $|\mathcal{E}| = |\G|^2$ and it is easy to see that $\mathcal{T}(\mathcal{E}) =  |\G| = |\mathcal{E}|^{1/2}$.
    It follows that the bound of Lemma \ref{l:TE_T} is tight.
\label{r:tight1}
\end{remark}

\bigskip

    Let us consider another characteristic of the set $S(\mathcal{E})$.

\begin{definition}
 Take
 the intersection of $\mathcal{E}$ with any  of three planes $\{ x=1\}, \{y=1\}, \{ z=1\}$, say with $\{ z=1\}$.
We obtain points $(a_j,b_j,1) \in \mathcal{E}$.
Then by $\mathcal{T}_* (\mathcal{E})$ denote the size of a maximal subset $J$ of $[|\mathcal{E}|]$
such that for any $j\in J$ either
$a_j \neq a_i$ or $b_j \neq b_i$ or $a_j b_j^{-1} \neq a_i b_i^{-1}$ for all $i\in J$, $i\neq j$.
In other words, each point $(a_j,b_j)$ has either a unique abscissa or ordinate or its ratio.
\label{def:T_*}
\end{definition}

Thus $\mathcal{T}_* (\mathcal{E}) \le |\mathcal{E}|$ and the bound it attained if, say,
$S(\mathcal{E}) = \{1\} \times \{ 1 \} \times [|\mathcal{E}|]$.
Now we obtain a lower bound for the quantity $\mathcal{T}_* (\mathcal{E})$.

\begin{lemma}
    We have $\mathcal{T}_* (\mathcal{E}) \ge 2|\mathcal{E}|^{1/2}-1$.
\label{l:TE}
\end{lemma}
\begin{proof}
Consider the intersection of $\mathcal{E}$ with any  of three planes $\{ x=1\}, \{y=1\}, \{ z=1\}$, say with $\{ z=1\}$ and
put $s=|S(\mathcal{E})| = |\mathcal{E} \cap \{ z=1\}|$.
    So we can think about $S(\mathcal{E})$ as a subset of $\F_p^* \times \F_p^*$.
    Take minimal sets $A,B\subseteq \F_p^*$ such that $S(\mathcal{E}) \subseteq A\times B$.
    Put $a=|A|$ and $b=|B|$. Then $s\le ab$.
	Clearly, there is $x_*$ such that $|\{ x=x_*\} \cap S(\mathcal{E})| \ge s/a$ and, similarly,
	there exists $y_*$ with  $|\{ y=y_*\} \cap S(\mathcal{E})| \ge s/b$.
    Put $V=\{ x=x_*\} \cap S(\mathcal{E})$ and $H=\{ y=y_*\} \cap S(\mathcal{E})$.
    If $V \cap H = \emptyset$ then each point in $V\cup H$ has either a unique abscissa or a unique ordinate.
    Now suppose that $V \cap H = (x_*, y_*)$. Then it is easy to see that the point $(x_*, y_*)$ has a unique ratio $x_*/y_*$
    differs from ratios of points in $V \cup H$.
    Thus any point in $V\cup H$ has either a unique abscissa or ordinate or its ratio.
    It gives us $\mathcal{T}_* (\mathcal{E}) \ge s/a+s/b-1$.
	Optimizing the expression $s/a+s/b-1$ over $a,b$ subject to $s\le ab$, we get
	$\mathcal{T}_* (\mathcal{E}) \ge 2|\mathcal{E}|^{1/2}-1$
as required.
$\hfill\Box$
\end{proof}

\begin{remark}
    Let $\G$ be a subgroup of $\F_p^*$ and $S(\mathcal{E}) = (\G \times \G \times \G)/ \sim$.
    Then we have $\G/\G=\G$ and hence  $\mathcal{T}_* (\mathcal{E}) \le  3|\G| = 3|\mathcal{E}|^{1/2}$
    (actually one can show that
    $$\mathcal{T}_* (\mathcal{E}) \ge \max_{X,Y \subseteq \G} (|X/Y| + 2 |\G| - |X| - |Y|-1)$$
    and thus a lower bound for $\mathcal{T}_* (\mathcal{E})$ is $(3-o(1))|\G|$. We do not need this fact.)
    It follows that the bound of Lemma \ref{l:TE} is tight up to constants.
\label{r:tight2}
\end{remark}

\bigskip

We say that $S(\mathcal{E})$ forms the Cartesian product if $S(\mathcal{E})$ is equivalent to the Cartesian product
$A\times B \times \{1\}$  or $A\times \{1\} \times B$ or $\{1 \} \times A\times B$ in two--dimensional projective plane.
With some abuse of the notation we write sometimes  $S(\mathcal{E}) = A\times B$ in this case.
Notice that always
\begin{equation}\label{f:trivial_T}
	\mathcal{T}_* (\mathcal{E}) \ge |\{ x/y ~:~ (x,y,1) \in S(\mathcal{E}) \}| \,.
\end{equation}
In particular, if $S(\mathcal{E}) = A\times B$, then  $\mathcal{T}_* (\mathcal{E}) \ge |A/B|$ but it is easy to see that
$\mathcal{T} (\mathcal{E}) = \max \{ |A|, |B| \}$.
Thus the quantities $\mathcal{T} (\mathcal{E}), \mathcal{T}_* (\mathcal{E})$ are incomparable  in general although we have a trivial inequality
$\mathcal{T} (\mathcal{E}) \le \mathcal{T}_* (\mathcal{E})$  of course.

\bigskip

Now we consider several examples of concrete systems of equations (\ref{f:E_form}).

\begin{example}
    Let $\Gamma \subseteq \F^*_p$ be a multiplicative subgroup.
    We are interested in basis properties of $\Gamma$, that is, in a question when $\Gamma + \Gamma$ contains $\F_p^*$.
    If $\Gamma + \Gamma$ does not contain $\F_p^*$, then it is easy to see that for some nonzero $\xi$ one has
    $(\Gamma + \Gamma) \cap \xi \Gamma = \emptyset$.
    It means that, taking any $a,b \in \G$ the equation $ax+by-\xi z = 0$ has no solutions in $\Gamma$.
    Thus $S(\mathcal{E})$ is the Cartesian product in this case.

    Similarly, one can consider a slightly general situation and study sets $A$ with $(A + A) \cap \Gamma = \emptyset$,
    where $A$ is not necessary $\G$--invariant.
    Here the equation $\gamma x+ \gamma y - z =0$, $x,y\in A$, $z\in \Gamma$ has no solutions for any $\gamma \in \Gamma$
    and thus $S(\mathcal{E}) = \{ (\gamma, \gamma, -1) ~:~ \gamma \in \G \}$ for corresponding triple $(A,A,\Gamma)$.
\label{ex:Gamma}
\end{example}


Proposition below shows that one cannot replace the constant $\kappa$ in Theorem \ref{t:main_intr} and $\kappa_1, \kappa_2, \kappa_3$
in Theorem \ref{t:main} below by something greater than $1/2$.

\bigskip

\begin{proposition}
	For any $k\ge 1$ there is a system $\mathcal{E}$ with
	$\mathcal{T}_* (\mathcal{E}) \gg |\mathcal{E}| \ge k$ 	
	such that for all sufficiently large $p\ge p(k)$
	there exists
	$A \subseteq \F_p$ avoiding the family $\mathcal{E}$
	with 	
	\begin{equation}\label{f:low_bounds}
		|A| \gg
    \frac{p}{|\mathcal{E}|^{1/2}}
		\,.
	\end{equation}
\label{p:low_bounds}
\end{proposition}
\begin{proof}
	Let
    $2 \le q < \sqrt{p}$
    be an even  parameter and
$$
	A := \{ 1 \le x < p/q ~:~ x \equiv 1 \pmod 2 \} \,.
$$
	We have $|A| \gg p/q$. 	
	Put $q'=q/2$ and $S = \{ (i,j) \in [q'] ~:~ i,j \equiv 0 \pmod 2 \}$.
	Clearly, $q^2 \ge |S| \gg q^2$, so taking $p$ and hence $q$ sufficiently large, we get $|S| \ge k$.
	Finally,  let $\mathcal{E}$ be the Cartesian product with $\{ z=1\}$ and let $S(\mathcal{E}) = -S$,
    so $\mathcal{E} \sim (-S) \times \{ 1\}$.
    Also let us
    put $d_j = 0$ in (\ref{f:E_form}).
    Once again,
    we have proved that $|\mathcal{E}| \ge k$ already.

	First of all, let us prove that $\mathcal{T}_* (\mathcal{E}) \gtrsim |\mathcal{E}| \gg q^2$.
    Consider the
    square--free numbers $F_0$ from $[q']$, notice that $|F_0| \gg q' \gg q$
	and easily check the identity $|F_0/F_0| = |F_0|^2 \gg q^2$.
	After that
    we take
    the set $S(\mathcal{E'}) := \{ (i,j) \in F_0^2 ~:~ (i,j) \in S\} \subseteq S(\mathcal{E})$
    and we get by (\ref{f:trivial_T}) and the fact $|F_0/F_0| = |F_0|^2$  that
    $\mathcal{T}_* (\mathcal{E}) \ge \mathcal{T}_* (\mathcal{E'}) \gg q^2$.

	Secondly, we need to check that $A$ avoiding the family $\mathcal{E}$.
	If not, then there are $x,y,z\in A$ such that $ix+jy \equiv z\pmod p$.
	We have $|ix+jy-z| < 2q'\cdot p/q = p$ and hence $ix+jy=z$.
	Thus, from $i,j \equiv 0 \pmod 2$, it follows that $z=ix + jy \equiv 0 \pmod 2$ but by the definition of the set $A$, we know that $z\equiv 1 \pmod 2$ which is a contradiction.
This completes the proof.
$\hfill\Box$
\end{proof}

\bigskip

\begin{remark}
    In \cite{Yekhanin}, \cite{AB} it was constructed a multiplicative subgroup $\G \subseteq \F_q \setminus \{ 0 \}$ with $|\G| \gg q^{2/3}$
    and having no solutions of the equation $x+y+z=0$.
    In view of Example \ref{ex:Gamma} it gives another
    system $\mathcal{E}$ such that (\ref{f:low_bounds}) holds in a general field $\F_q$.
    Here $\mathcal{T} (\mathcal{E}) \sim \mathcal{T}_* (\mathcal{E}) \sim |\G| = |\mathcal{E}|^{1/2}$, see Remarks \ref{r:tight1}, \ref{r:tight2}.

    If we consider the case of prime $p$, then
    a multiplicative subgroup $\G \subseteq \F_p^*$ without solutions of the equation $x+y=z$ was constructed with the constraint $|\G| \gg p^{1/3}$ only,    see \cite{AB}.
    Conjecturally, the right bound here is $|\G| \ge p^{1/2-o(1)}$ or even $p^{2/3-o(1)}$, see \cite[Section 4]{AB}.
\end{remark}


    We finish this section considering another
    family $\mathcal{E}$.

\begin{example}
    Take
    the family of equations
    \begin{equation}\label{f:Sar}
        x+y+\lambda z = 0 \,,
    \end{equation}
    where $\lambda \in \Lambda$ and $\Lambda \subseteq \F_p^*$ is a set.
    We have $t:= |\Lambda| = \mathcal{T} (\mathcal{E}) = \mathcal{T}_* (\mathcal{E}) = |\mathcal{E}|$ for this  family.
    By the main result from \cite{Sarkozy_eq} the equation
    $$
        a+b=cd \,, \quad \quad a\in A,\, b\in B,\, c\in C,\, d\in D
    $$
            has a solution if $|A||B||C||D| \gg p^3$.
    In other words, if $A$ avoids all equations (\ref{f:Sar}), then $|A| \ll p / t^{1/3}$.
    The same
    bound holds in the case of the Cartesian product (to see this just fix a variable, say, $b$ in the correspondent equation $a+bc+de=0$.)
    An improvement of $1/3$ would imply  a considerable progress in the area (in particular, for basis properties of multiplicative subgroups) and seems unattainable at the moment.
\label{ex:Sarkozy}
\end{example}

    Further examples of families $\mathcal{E}$ can be found in the last section \ref{sec:applications}.

\section{Preliminaries}
\label{sec:preliminaries}

Let us begin with a simple lemma about the triangle inequality for restricted energies.

\begin{lemma}
	For any four sets $A,B,X,Y \subseteq \Gr$ one has
	$$
		\E_A (X \sqcup Y,B)^{1/2} \le \E_A (X,B)^{1/2} + \E_A (Y,B)^{1/2} \,.
	$$
\label{l:norm_E+}
\end{lemma}
\begin{proof}
	We have
$$
	\E_A (X \sqcup Y,B) = \sum_{a\in A} ((X \c B)(a) + (Y\c B) (a))^2 =
$$
$$
	= \E_A (X,B) + \E_A (Y,B) + 2\sum_{a\in A} (X\c B) (a) (Y\c B) (a) \,.
$$
	By the Cauchy--Schwarz inequality, we get
$$
	\E_A (X+Y,B) \le \E_A (X,B) + \E_A (Y,B) + 2 \E^{1/2}_A (X,B) \E^{1/2}_A (Y,B) = (\E^{1/2}_A (X,B) + \E^{1/2}_A (Y,B))^2
$$
    as required.
$\hfill\Box$
\end{proof}

\bigskip

Now we need in some sum--product results.
In \cite{AMRS} it was shown that the main
theorem
from \cite{misha} implies the following weaker version of the Szemer\'edi--Trotter
Theorem
for a general field $\F$.

\begin{theorem}\label{t:SzT_Fp}
Let $A, B\subset \F$  be two sets, with $|B|\le |A|\leq p^{2/3}$ in positive characteristic.
The number of incidences between the point set $A\times B$ and  any set of $m$ lines in $\F^2$ is
$$O((|A||B|)^{3/4} m^{2/3}+m + |A||B|) \,.$$
\end{theorem}

In \cite{AMRS} (see also \cite{RRS}) it was obtained another particular sum--product result which we will use in the next sections.
For more general context
consult with paper
\cite{AMRS}.

\begin{theorem}
    Suppose that $A,B,C\subseteq \F_p$ are sets with $|A||B||C| = O(p^2)$.
    Then
$$
    |\{(a_1,a_2,b_1,b_2,c_1,c_2)\in A^2\times B^2\times C^2\,:\,a_1(b_1+c_1)=a_2(b_2+c_2) \}|
    \ll
$$
$$
    \ll (|A| |B| |C|)^{\frac32}+ |A| |B| |C| \max \{ |A|, |B|, |C| \} \,.
$$
\label{t:RRS}
\end{theorem}


Let us obtain  a  consequence of Theorem \ref{t:SzT_Fp} in the spirit of paper \cite{KS2}.

\begin{lemma}
	Let $A \subset \F_p$ be a set such that $A\subseteq \Sym^{+}_t (P,Q)$, where $P,Q\subseteq \F_p$ are two another sets
	with $|P|^2 |Q|^{5/4} \ge t^{2} |A|^{3/4}$ and $|Q| < p^{2/3}$.
	Then for any set $B$, $|B| < p^{2/3}$ one has
	\begin{equation}\label{f:D_Fp_pre}
		|\{ s ~:~ |A\cap Bs| \ge \tau \}| \ll \frac{|P|^2 |Q|^{9/4} |B|^{9/4}}{t^3 \tau^3} \,.
	\end{equation}
\label{l:D_Fp_pre}
\end{lemma}
\begin{proof}
	Let $S_\tau$ be the set in the left--hand side  of (\ref{f:D_Fp_pre}).
	    We have
\begin{equation*}\label{tmp:02.12.2015_1}
    \tau |S_\tau| \le \sum_{s\in S_\tau} |A\cap Bs| = |\{ a= bs ~:~ a\in A,\, b\in B,\, s\in S_\tau \}|
        := \sigma \,.
\end{equation*}
	Because $A\subseteq \Sym^{+}_t (P,Q)$, we obtain the following upper bound for the number of solutions $\sigma$
\begin{equation}\label{tmp:03.12.2015_1-}
    \sigma \le t^{-1} |\{ p+q = sb ~:~ p\in P,\, q\in Q,\,  b\in B,\, s\in S_\tau \}| \,.
\end{equation}
	First of all, let us prove a trivial estimate for the size of $S_\tau$.
Namely, dropping the condition $s\in S_\tau$ in (\ref{tmp:03.12.2015_1-}),
we get
$$
    \tau |S_\tau| t \le |P| |Q| |B|
$$
and hence inequality (\ref{f:D_Fp_pre}) should be checked in the range
\begin{equation}\label{f:t_tau_range}
    t^2 \tau^2 \gg |P| |Q|^{5/4} |B|^{5/4}
\end{equation}
only because otherwise
$$
    |S_\tau| \le \frac{|P||Q||B|}{t\tau} \ll \frac{|P|^2 |Q|^{9/4} |B|^{9/4}}{t^3 \tau^3} \,.
$$
Let us notice one consequence of (\ref{f:t_tau_range}).
	Using inequality (\ref{f:t_tau_range}) again as well as trivial bounds
$\tau \le |B|$ and $t\le |Q|$, we have
\begin{equation}\label{f:t_tau_range++}
	|P|\ll |Q|^{3/4} |B|^{3/4} \,.
\end{equation}

Further consider the family $\mathcal{L}$ of $|P||S_\tau|$ lines $l_{p,s} = \{ (x,y) ~:~  p+y = sx \}$, $p\in P$, $s\in S_\tau$
and the family of points $\mathcal{P} = B \times Q$.
By our assumptions, we have $|B|, |Q| < p^{2/3}$.
Applying Theorem \ref{t:SzT_Fp} to
the pair $(\mathcal{P}, \mathcal{L})$,
we get
\begin{equation}\label{tmp:03.12.2015_1}
    \sigma \le t^{-1} \mathcal{I} (\mathcal{P}, \mathcal{L})
        \ll
            t^{-1} \left( (|\mathcal{P}|^{3/4} |\mathcal{L}|^{2/3} + |\mathcal{P}| + |\mathcal{L}| \right) \,.
\end{equation}
If the first term in (\ref{tmp:03.12.2015_1}) dominates, then we obtain (\ref{f:D_Fp_pre}).
Now suppose that required bound (\ref{f:D_Fp_pre}) does not hold.
Then if the second term in (\ref{tmp:03.12.2015_1})
is the largest one,
we obtain
$$
    \frac{|P|^2 |Q|^{9/4} |B|^{9/4}}{t^2 \tau^2} \ll t \tau |S_\tau| \ll |\mathcal{P}| = |Q| |B| \,.
$$
But, clearly, $t \le \min \{ |P|, |Q|\}$ and $\tau \le \min \{ |A|, |B| \}$,
thus
$$
    |P|^2 |Q|^{5/4} \ll t^2 |A|^{3/4}
$$
and
we arrive to a contradiction
with our assumption
$|P|^2 |Q|^{5/4} \ge t^{2} |A|^{3/4}$
(actually, we need in $|P|^2 |Q|^{5/4} \gg t^{2} |A|^{3/4}$ but we can assume just $|P|^2 |Q|^{5/4} \ge t^{2} |A|^{3/4}$ increasing the constants in our main bound (\ref{f:D_Fp_pre})).
Finally, we need to consider the case when the third term in (\ref{tmp:03.12.2015_1}) dominates.
In this  situation
$$
    t \tau |S_\tau| \ll |\mathcal{L}| = |S_\tau| |P|
$$
and hence in view of (\ref{f:t_tau_range})
\begin{equation}\label{tmp:02.08.2016_1}
    |P| |Q|^{5/4} |B|^{5/4} \ll |P|^2 \,.
\end{equation}
Recalling (\ref{f:t_tau_range++}), we obtain
$$
    |Q|^{5/4} |B|^{5/4} \ll |Q|^{3/4} |B|^{3/4} \,.
$$
But this is a contradiction
if $Q$ or $B$ is large enough.
    This completes  the proof.
$\hfill\Box$
\end{proof}

\bigskip

By a simple summation we get an immediate consequence of the last lemma.

\begin{corollary}
	Let $A \subset \F_p$ be a set such that $A\subseteq \Sym^{+}_t (P,Q)$, where $P,Q\subseteq \F_p$ are two another sets
	with
    $|P|^2 |Q|^{5/4} \ge t^{2} |A|^{3/4}$ and $|Q| < p^{2/3}$.	
Then for any sets $X,Y$, $|Y| < p^{2/3}$ one has
$$
	\E^{\times}_X (A,Y) \ll \frac{|P|^{4/3} |Q|^{3/2} |Y|^{3/2} |X|^{1/3}}{t^2} \,.
$$
\label{c:D_Fp_pre}
\end{corollary}

Very recently, after submitting this paper on the arXiv S. Stevens and F. de Zeeuw obtained a new incidence bound in $\F_p$, see paper \cite{SZ_inc}.

\begin{theorem}
    Let $A, B\subset \F$  be two sets.
The number of incidences between the point set $A\times B$ and  any set of $n$ lines in $\F^2$ is
$$O( |B|^{1/2} |A|^{3/4} n^{3/4} + |A|^{1/2} |B| n^{1/2}) \,,$$
provided $|A| n \ll p^2$. \\
In particular, if $|A|\ge |B|$ and $n\ge |B|$, then the number of incidences is $O( |B|^{1/2} |A|^{3/4} n^{3/4})$,
provided $|A| n \ll p^2$.
\label{t:S_Z_inc}
\end{theorem}

It gives us variants  of Lemma \ref{l:D_Fp_pre'} and Corollary \ref{c:D_Fp_pre}.

\begin{lemma}
	Let $A \subset \F_p$ be a set such that $A\subseteq \Sym^{+}_t (P,Q)$, where $P,Q\subseteq \F_p$ are two another sets.
	Then for any set $B$, $|B| |P| |Q| \ll p^{2}$ one has
	\begin{equation}\label{f:D_Fp_pre'}
		|\{ s ~:~ |A\cap Bs| \ge \tau \}| \ll \frac{(|P| |Q| |B|)^{3/2}}{(t \tau)^2} \,.
	\end{equation}
\label{l:D_Fp_pre'}
\end{lemma}

\begin{corollary}
	Let $A \subset \F_p$ be a set such that $A\subseteq \Sym^{+}_t (P,Q)$, where $P,Q\subseteq \F_p$ are two another sets.
Then for any set $Y$, $|Y| |P| |Q| \ll p^{2}$
one has
$$
	\E^{\times} (A,Y) \lesssim \frac{(|P||Q||Y|)^{3/2}}{t^{2}} \,.
$$
\label{c:D_Fp_pre'}
\end{corollary}

We use Corollary \ref{c:D_Fp_pre} in the proof of the third part of forthcoming Proposition \ref{p:technical_mc} only and hence in the third part in Theorem \ref{t:main}.
Replacing this Corollary onto Corollary \ref{c:D_Fp_pre'}, one can obtain a slightly better result.
We do not make such calculations.


\section{The multiplicative energy of the spectrum}
\label{sec:energy}

Now recall the notion of the spectrum $\Spec_\eps (A)$ of a set $A$ and formulate the required result about the structure of
$\Spec_\eps (A)$.
Let $A\subseteq \F_p$ be a set, and $\eps \in (0,1]$ be a real number.
Define
$$
    \Spec_\eps (A) = \{ r \in \F_p ~:~ |\FF{A}(r)| \ge \eps |A| \} \,.
$$
Clearly, $0\in \Spec_\eps (A)$, and $\Spec_\eps (A) = - \Spec_\eps (A)$.
In this section we denote by $\d$ the density of our set $A$, that is, $\d = |A|/p$.
From Parseval identity (\ref{F_Par}), we have a simple upper bound for the size of the spectrum, namely,
\begin{equation}\label{f:spec_Par}
    |\Spec_\eps (A)| \le \frac{p}{|A| \eps^2}  = \frac{1}{\d \eps^2} \,.
\end{equation}

We need in a result from \cite{S_les}  (tight bounds are contained in paper \cite{S_diss})
which shows that the spectrum has a rich additive structure.

\begin{lemma}
    Let $A \subset \F_p$ be a set, $k\ge 2$ be an integer, and $\eps \in (0,1]$ be a real number.
    Then for any $B\subseteq \Spec_\eps (A)$ one has
\begin{equation}\label{f:les}
    \T^+_k (B) \ge \eps^{2k} |B|^{2k} \cdot |A|/p \,.
\end{equation}
\label{l:les}
\end{lemma}

\bigskip

We need in  Proposition 16 from \cite{RSS} and a combinatorial lemma which is contained  in the proof of this proposition.
We give the  proof of this  lemma for completeness.

\begin{lemma}
    Let $A\subseteq \Gr$ be a set, $P\subseteq A-A$, $P=-P$.
    Then there is $A_* \subseteq A$ and a number $q$, $q \lesssim |A_*|$
    such that for any $x\in A_*$ one has $(A*P) (x) \ge q$,
     and $\sigma_P (A) \sim |A_*| q$.
\label{l:Misha_c}
\end{lemma}
\begin{proof}
    We have
    $$
        \sigma := \sigma_P (A) = \sum_{x\in P} (A\c A) (x) = \sum_{x\in A} (A*P)(x) \,.
    $$
    Using the pigeonhole principle, we find $A'\subseteq A$ such that $(A*P) (x)$
    differ by a multiplicative factor of at most twice  on
    $A'$
    and $\sigma \sim q' |A'|$, where $q' = \min_{x\in A'} (A*P)(x)$.
    If $q'\le |A'|$, then put $A_* = A'$, $q=q'$ and we are done. Suppose not.
    By assumption $P=-P$, and thus we get
    $$
        \sigma \lesssim \sum_{x\in A'} (A*P)(x) = \sum_{x\in A} (A'*P)(x) \,.
    $$
    Then applying the pigeonhole principle one more time, we find $A''\subseteq A$ such that $(A'*P) (x)$ differ by a multiplicative factor
    of at most twice on $A''$ and $\sigma \sim q'' |A''|$, where $q'' = \min_{x\in A''} (A'*P)(x)$.
    Using the inequality  $q' > |A'|$ and a trivial bound $q'' \le |A'|$, we obtain
    $$
        |A''| |A'| \ge |A''| q'' \gtrsim \sigma \ge |A'| q' > q'' |A'|
    $$
    and hence $q'' \lesssim |A''|$.
    After that we put $A_* = A''$, $q=q''$.
    This completes  the proof.
$\hfill\Box$
\end{proof}

\bigskip

Recall a result from \cite{RSS}, see Proposition 16 from here.

\begin{proposition}
    Let $S \subseteq \F_p$ be a set, $|S|^6 \lesssim p^2 \E^{\times} (S)$.
    Then there is a set $S_1 \subseteq S$, $|S_1|^2 \gtrsim \E^{\times} (S)/ |S|$ and
\begin{equation}\label{f:E*+}
    \E^{+} (S_1)^2 \E^{\times} (S)^3 \lesssim |S_1|^{11} |S|^3 \,.
\end{equation}
    The same result holds if one replace $+$ onto $\times$ and vice versa.
\label{p:E*+}
\end{proposition}

Now we are ready to prove that
any (large) subset of the spectrum
is always has small multiplicative energy.
This is one of the  main results of this section.

\begin{theorem}
    Let $A \subset \F_p$ be a set,  and $\eps \in (0,1]$ be a real number.
    Then for any $B\subseteq \Spec_\eps (A)$, $|B| < \d^{-1/6} \eps^{-2/3} \sqrt{p}$, one has
\begin{equation}\label{f:energy*spec}
    \E^{\times} (B) \lesssim |B|^2 \cdot \d^{-2/3} \eps^{-8/3} \,.
\end{equation}
\label{t:energy*spec}
\end{theorem}
\begin{proof}
	If $\E^{\times} (B) \lesssim |B|^2 \cdot \d^{-2/3} \eps^{-8/3}$, then it is nothing to prove.
	Otherwise, in view of our assumption, we have $|B|^6 \lesssim p^2 \E^{\times} (B)$.
    Applying Proposition \ref{p:E*+} with $S=B$, we find $B_1 \subseteq B$ such that
    $|B_1|^2 \gtrsim \E^{\times} (B)/ |B|$  and
$$
    \E^{+} (B_1)^2 \E^{\times} (B)^3 \lesssim |B_1|^{11} |B|^3 \,.
$$
    Further, using Lemma \ref{l:les} for $B=B_1$ and $k=2$, we have
    $\E^{+} (B_1) \ge \d \eps^4 |B_1|^4$.
    Thus
$$
    \E^{\times} (B)^3 \lesssim |B_1|^3 |B|^3 \d^{-2} \eps^{-8}  \le |B|^6 \d^{-2} \eps^{-8}
$$
as required.
$\hfill\Box$
\end{proof}


\begin{example}
    Let $\eps \gg 1$, $B = \Spec_\eps (A)$, $|A| \gg p^{2/5}$ and the size of $B$ is comparable with the upper bound which is given by (\ref{f:spec_Par}),
    namely, $|B| \gg \d^{-1}$.
    Then $\E^{\times} (B) \lesssim |B|^{8/3}$.
    It means that we have a non--trivial estimate for the multiplicative energy of the spectrum in this case.
\end{example}


\begin{remark}
    The proof of Theorem \ref{t:energy*spec} shows that
    a similar
    statement about the energy holds for wider class of so--called {\it connected} sets, that is, sets $S$ such that for any $S' \subseteq S$, $|S'| \gg |S|$ one has $\E (S') \gg \E (S)$, see the rigorous definition in  \cite{s_energy}, say.
\end{remark}


It is possible to increase the size of the set $S_1$ in Proposition \ref{p:E*+} decreasing the upper estimate for the product of energies in (\ref{f:E*+}).
Our arguments  mimic the proof of Corollary 22 from \cite{KS2}.

\begin{corollary}
	Let $S \subseteq \F_p$ be a set, $|S|^6 \lesssim p^2 \E^{\times} (S)$.
    Then there is a set $S' \subseteq S$, $|S'|^3 \gtrsim \E^{\times} (S)$ and
\begin{equation}\label{f:E*+_new}
    \E^{+} (S')^2 \E^{\times} (S)^3 \lesssim |S|^{14} \,.
\end{equation}
    The same result holds if one replace $+$ onto $\times$ and vice versa.
\label{c:E^*+_new}
\end{corollary}
\begin{proof}
	Our arguments is a sort of an algorithm.
We construct a decreasing sequence of sets $U_1=S \supseteq U_2 \supseteq \dots \supseteq U_k$ and an increasing sequence of sets $V_0 = \emptyset \subseteq V_1 \subseteq \dots \subseteq V_{k-1} \subseteq S$ such that
    for any $j=1,2,\dots, k$ the sets $U_j$ and $V_{j-1}$  are disjoint and moreover $S = U_j \sqcup V_{j-1}$.
    If at some step $j$ we have $|V_j| > (\E^\times (S))^{1/3} / 2$, then we stop our algorithm putting
    $S'=V_j$
    and $k=j-1$.
    In the opposite situation we have $|V_j| \le (\E^\times (S))^{1/3} / 2$.
    Applying Proposition \ref{p:E*+} to the set $U_j$,
    we find
    the subset $Y_j$ of $U_j$ such that $|Y_j|^2 \gtrsim \E^\times (U_j) / |U_j|$ and  such that
$$
    \E^{+} (Y_j)^2 \E^{\times} (U_j)^3 \lesssim |Y_j|^{11} |U_j|^3 \,,
$$
	provided  $|U_j|^6 \lesssim p^2 \E^\times (U_j)$.
	Now notice that the inequality $|V_j| \le (\E^\times (S))^{1/3} / 2$ implies that
    $\E^\times (V_j) \le |V_j|^3 \le \E^\times (S) /8$
	and hence $\E^\times (U_j) \gg \E^\times (S)$.
	In particular, $|Y_j|^2 \gtrsim \E^\times (S) / |S|$ and,
	further
$$
	|U_j|^6 \le |S|^6 \lesssim p^2 \E^\times (S) \ll p^2 \E^\times (U_j) \,,
$$
	and thus the condition  $|U_j|^6 \lesssim p^2 \E^\times (U_j)$ takes place.
    Hence
$$
	    \E^{+} (Y_j)  \lesssim |Y_j|^{11/2} |S|^{3/2} \E^{\times} (S)^{-3/2} \,.
$$
	After that we put $U_{j+1} = U_j \setminus Y_j$, $V_j = V_{j-1} \sqcup Y_j$ and
	repeat the procedure.
    Clearly, for each number $k$, we have $V_k = \bigsqcup_{j=1}^k Y_j$ and
	it is easy to see that our algorithm must stop at some step $k$.
	Put $S'=V_{k-1}$.	
    It is known that $\E^{1/4} (\cdot)$ is a norm, see, e.g., \cite{Tao_Vu_book} or \cite{KS2}, say
    (this fact can be considered as an analog of Lemma \ref{l:norm_E+} with $A=\F_p$ as well),
    whence
	$$
		\E^+ (S') \le \left( \sum_{j=1}^{k-1} \E^{+} (Y_j)^{1/4} \right)^4
			\lesssim
				 |S|^{3/2} \E^{\times} (S)^{-3/2} \left( \sum_{j=1}^{k-1} |Y_j|^{11/8} \right)^4
				 	\le
	$$
	$$
				 	\le
				 		|S|^{3/2} \E^{\times} (S)^{-3/2} |S|^{11/2}
				 			=
				 				|S|^7 \E^{\times} (S)^{-3/2}
	$$
as required.
$\hfill\Box$
\end{proof}

\bigskip

Using the
corollary
above we can prove that any subset of the spectrum has large subset with small multiplicative energy.

\begin{theorem}
    Let $A \subset \F_p$ be a set,  and $\eps \in (0,1]$ be a real number.
    Then for any $B:=\Spec_\eps (A)$, $|B| < \d^{1/2} \eps^2 p$
    there is $B' \subseteq B$ such that $|B'|^3 \gtrsim \d \eps^4 |B|^4$ and
\begin{equation}\label{f:energy*spec_new}
    \E^{\times} (B')  \lesssim |B| \cdot \d^{-3/2} \eps^{-6} \,.
\end{equation}
\label{t:energy*spec_new}
\end{theorem}
\begin{proof}
    Suppose that $|B|^6 \lesssim p^2 \E^{+} (B)$.
    Then we use Corollary \ref{c:E^*+_new}, reversing $+$ onto $\times$.
    Thus, there is $B'\subseteq B$, $|B'|^3 \gtrsim \E^{+} (B)$ and
\begin{equation}\label{tmp:11.08.2016_1}
    \E^\times (B')^2 \E^{+} (B)^3 \lesssim |B|^{14} \,.
\end{equation}
    Applying Lemma \ref{l:les} with $k=2$, we see that $\E^{+} (B) \ge \d \eps^4 |B|^4$.
    It gives us, firstly, $|B'|^3 \gtrsim \eps \d^4 |B|^4$ and, secondly, from (\ref{tmp:11.08.2016_1}), it follows that
$$
    \E^\times (B')^2 \d^3 \eps^{12} |B|^{12} \lesssim |B|^{14}
$$
as required.
    To check inequality $|B|^6 \lesssim p^2 \E^{+} (B)$, we recall that $\E^{+} (B) \ge \d \eps^4  |B|^4$.
    This completes the proof.
$\hfill\Box$
\end{proof}


\begin{example}
    Let $\eps \gg 1$, $B = \Spec_\eps (A)$, $|A| \gg p^{1/3}$ and the size of $B$ is comparable with upper bound (\ref{f:spec_Par}),
    namely, $|B| \gg \d^{-1}$.
    Then by Theorem \ref{t:energy*spec_new} we find a set $B'\subseteq B$ such that $\E^{\times} (B') \lesssim |B'|^{5/2}$
    and $|B'| \gtrsim |B|$.
\label{ex:5/2_2}
\end{example}


Theorem \ref{t:energy*spec_new} immediately implies

\begin{corollary}
    Let $A \subset \F_p$ be a set,  and $\eps \in (0,1]$ be a real number.
    Then for any $B:=\Spec_\eps (A)$, $|B| < \d^{1/2} \eps^2 p$
    there is $\tilde{B} \subseteq B$ such that $|\tilde{B}| \ge |B|/2$ and
$$
    \E^{\times} (\tilde{B})  \lesssim \d^{-17/6} \eps^{-34/3} |B|^{-1/3}
        \le
            \d^{-5/2} \eps^{-32/3} \,.
$$
\label{c:energy*spec_new}
\end{corollary}
\begin{proof}
    Applying Theorem \ref{t:energy*spec_new} to the set $B$, we find  $B'_1 := B' \subseteq B$ such that (\ref{f:energy*spec_new}) holds and
    $|B'_1|^3 \gtrsim \d \eps^4 |B|^4$.
    Consider $B\setminus B'_1$. If $|B\setminus B'_1| < |B|/2$,  then we are done.
    If not, then apply the same arguments to this set.
    An so on.
    At the end we have constructed a sequence of disjoint subsets of $B$, namely, $B'_1, \dots, B'_k$ such that the set $\tilde{B} := \bigsqcup_{j=1}^k B'_j$
    has size at least $|B|/2$.
    Clearly, $k \lesssim \d^{-1/3} \eps^{-4/3} |B|^{-1/3}$.
    Because $\E^{1/4} (\cdot)$ is a norm, we obtain in view of (\ref{f:energy*spec_new})
    and the Parseval identity that
$$
    \E^\times (\tilde{B}) \le k^4 |B| \d^{-3/2} \eps^{-6} = \d^{-17/6} \eps^{-34/3} |B|^{-1/3}
        \le
            \d^{-5/2} \eps^{-32/3} \,.
$$
    This completes the proof.
$\hfill\Box$
\end{proof}

\bigskip

As in Example \ref{ex:5/2_2} if $\eps \gg 1$, $B = \Spec_\eps (A)$, $|A| \gg p^{1/3}$ and  $|B| \gg \d^{-1}$,
then  we find a set $\tilde{B}\subseteq B$ such that $\E^{\times} (\tilde{B}) \lesssim |\tilde{B}|^{5/2}$ and $|\tilde{B}| \ge |B|/2$.

\section{The proof of the main result}
\label{sec:proof}

Using the results of the previous two parts of our paper,
we are ready to formulate the main technical proposition of this section.

\begin{proposition}
    Let $A \subset \F_p$ be a set, $\d = |A|/p$, and $\eps \in (0,1]$ be a real number.
    Then for an arbitrary  $B\subseteq \Spec_\eps (A)$, $|B| < \d^{-1/6} \eps^{-2/3} \sqrt{p}$ and any sets $C,D \subseteq \F_p$, one has
\begin{equation}\label{f:technical_mc1}
    \sum_{x\in D} (B \otimes C) (x)
        \lesssim
            \d^{-1/6} \eps^{-2/3} |B|^{1/2}
                   \min\{
                   |D|^{1/2} (\E^\times (C))^{1/4} ,
                   |C|^{1/2} (\E^\times (D))^{1/4}
                \} \,.
\end{equation}
Further suppose that $|C|\le |D|$ and
\begin{equation}\label{cond:technical_mc}
	|B| \le \d \eps^4 |C|^3 |D|^2 \le |B|^9
\end{equation}
    as well as
\begin{equation}\label{cond:technical_mc_add}
     \d^{-1/4} \eps^{-1} |B|^{9/4} |D|^{-1/2} |C|^{5/4} < p^2 \,.
\end{equation}
	Then
\begin{equation}\label{f:technical_mc2}
    \sum_{x\in D} (B \otimes C) (x)
        \lesssim
            \d^{-3/16} \eps^{-3/4} |B|^{11/16} |D|^{1/8} |C|^{15/16} \,.
\end{equation}
Finally, assuming $|C| \le |D| < p^{2/3}$, $|B| < \min\{ \d^{-1/6} \eps^{-2/3} \sqrt{p}, p^{2/3} \}$ and
\begin{equation}\label{cond:technical_mc_add'}
 |C|^{20} \le |D|^{45} \d^{15} \eps^{60} |B|^{24} \,,
\end{equation}
we get
\begin{equation}\label{f:technical_mc3}
    \sum_{x\in C} (B \otimes D)^2 (x)
        \lesssim
        	\min\{ \d^{-1/3} \eps^{-4/3} |B| (\E^{\times} (D))^{1/2},  \d^{-4/11} \eps^{-16/11} |C|^{9/11} |B|^{29/22} |D|^{9/22} \} \,.
\end{equation}
\label{p:technical_mc}
\end{proposition}
\begin{proof}
Let $\sigma = \sum_{x\in D} (B \otimes C) (x)$.
Using the Cauchy--Schwarz inequality twice, combining with Theorem \ref{t:energy*spec}, we get
$$
    \sigma^4 \le |D|^2 \E^{\times} (B) \E^\times (C)
        \lesssim
        |D|^2 \E^\times (C) |B|^2 \cdot \d^{-2/3} \eps^{-8/3}
$$
and bound (\ref{f:technical_mc1}) has proved.

Now let us prove (\ref{f:technical_mc2}).
First of all, notice that, trivially,  $\sigma \le |B| |C| $ and hence we can suppose
$$
 	|B| |C| >
 		\d^{-3/16} \eps^{-3/4} |B|^{11/16} |D|^{1/8} |C|^{15/16}
$$
or, in other words,
\begin{equation}\label{tmp:trivial_bc}
	|B|^5 |C| \eps^{12} \d^3 > |D|^2 \,.
\end{equation}
Finally, because of $B \subseteq \Spec_\eps (A)$ we have in view of (\ref{f:spec_Par})
\begin{equation}\label{tmp:Par_B}
	|B| \le \d^{-1} \eps^{-2} \,.
\end{equation}
Now let $M\ge 1$ be a parameter which we will choose later.
Our arguments is a sort of an algorithm.
We construct a decreasing sequence of sets $U_1=B \supseteq U_2 \supseteq \dots \supseteq U_k$ and an increasing sequence of sets $V_0 = \emptyset \subseteq V_1 \subseteq \dots \subseteq V_{k-1} \subseteq B$ such that
    for any $j=1,2,\dots, k$ the sets $U_j$ and $V_{j-1}$  are disjoint and moreover $B = U_j \sqcup V_{j-1}$.
    If at some step $j$ we have $\E^{+} (U_j) \le |B|^3 / M$, then we stop our algorithm putting
    $U=U_j$, $V = V_{j-1}$, and $k=j-1$.
    In the opposite situation we have $\E^{+} (U_j) > |B|^3 / M$.
    Using the pigeonhole principle we find a set $P_j \subseteq U_j-U_j$
    such that $\E^{+}_{P_j} (U_j) \sim \E^{+} (U_j)$ and a number $t=t_j$ with
    $t< (U_j \c U_j) (x) \le 2t$ for all $x\in P_j$.
    Applying Lemma \ref{l:Misha_c} to the sets $U_j$, $P_j$,
    we get
    the subset $Y_j$ of $U_j$ such that $|Y_j| \gtrsim |B|/M^{}$ and a number $q_j \lesssim |Y_j|$
    such that for any $x\in Y_j$ one has $(U_j * P_j)(x) \ge q_j$ and  $\sigma_{P_j} (U_j) \sim |Y_j| q_j$.
    After that we put $U_{j+1} = U_j \setminus Y_j$, $V_j = V_{j-1} \sqcup Y_j$ and repeat the procedure.
    Clearly, $V_k = \bigsqcup_{j=1}^k Y_j$ and because of $|Y_j| \gtrsim |B|/M^{}$, we have $k\lesssim M^{}$,
    so the number of steps is finite.

    Consider $\sigma_j = \sum_{x\in D} (Y_j \otimes C) (x)$.
    By the Cauchy--Schwarz inequality and the fact that $(U_j * P_j)(x) \ge q_j$ on $Y_j$, we have
$$
    \sigma^2_j \le |D| \E^{\times} (Y_j,C)
        \le
            q_j^{-2} |D| |\{ (u+p) c = (u'+p') c' ~:~  u,u'\in U_j,\, p,p' \in P_j,\, c,c' \in C \}| \,.
$$
    Applying
    Theorem \ref{t:RRS},
    we get
\begin{equation}\label{f:AMRS_1}
	\sigma^2_j \ll q_j^{-2} |D| ( (|U_j| |P_j| |C| )^{3/2} + |U_j| |P_j| |C| \cdot \max\{ |U_j|, |P_j|, |C| \}) \,,
\end{equation}
    provided
\begin{equation}\label{cond:later}
    |U_j| |P_j| |C| \ll p^2 \,.
\end{equation}
    We will check condition (\ref{cond:later}) later.
	Moreover suppose that the first term in (\ref{f:AMRS_1}) dominates.
 	Then using the fact $q_j |Y_j| \sim t |P_j|$, we obtain
$$
    \sigma^2_j \ll q_j^{-2} |D| ( (|U_j| |P_j| |C| )^{3/2}
        \lesssim |D| |C|^{3/2} |B|^{3/2} |Y_j|^2 t^{-2} |P_j|^{-1/2} \,.
$$
    Now recalling that $q_j \lesssim |Y_j|$ and observing
$$
    t |Y_j|^2 \gtrsim t q_j |Y_j| \sim t \sigma_{P_j} (U_j) \sim t^2 |P_j| \sim \E^{+} (U_j)  \,,
$$
    we get from $t |Y_j|^2 \gtrsim t^2 |P_j|$ and $t |Y_j|^2 \gtrsim \E^{+} (U_j)$
    that $\E^{+} (U_j) / |Y_j|^2 \lesssim t\lesssim |Y_j|^2 / |P_j|$ and
    hence in view of $t^2 |P_j| \sim \E^{+} (U_j)$, we derive
\begin{equation}\label{f:P_j_size}
	|P_j| \lesssim |Y_j|^4 / \E^{+} (U_j)
\end{equation}
and
$$
    \sigma^2_j
        \lesssim |D| |C|^{3/2} |B|^{3/2} |Y_j|^2 \E^{+} (U_j)^{-1} |P_j|^{1/2}
            \lesssim
                |D| |C|^{3/2} |B|^{3/2} |Y_j|^4 \E^{+} (U_j)^{-3/2}
                    \le
$$
$$
                    \le
                        M^{3/2} |D| |C|^{3/2} |Y_j|^4 |B|^{-3} \,.
$$
Thus
\begin{equation}\label{f:split_1}
    \sigma = \sum_{x\in D} (U \otimes C) (x) +  \sum_{x\in D} (V \otimes C) (x)
        \le
            \sum_{x\in D} (U \otimes C) (x) + \sum_{j=1}^k \sigma_j
                \lesssim
\end{equation}
\begin{equation}\label{f:split_1'.5}
                \lesssim
                    \sum_{x\in D} (U \otimes C) (x) + M^{3/4} |D|^{1/2} |C|^{3/4} |B|^{-3/2}\sum_{j=1}^l |Y_j|^2
                        \le
\end{equation}
\begin{equation}\label{f:split_1'}
                        \le
                            \sum_{x\in D} (U \otimes C) (x) + M^{3/4} |D|^{1/2} |C|^{3/4} |B|^{1/2} \,.
\end{equation}
To estimate the first term in the last formula, we remind that $\E^{+} (U) \le |B|^3/M$ and
$U\subseteq B \subseteq \Spec_\eps (A)$.
Using Lemma \ref{l:les}, we see that
$$
    \d \eps^4 |U|^4 \le \E^{+} (U) \le |B|^3/M \,.
$$
Whence
\begin{equation}\label{f:U_bound}
    |U| \le \d^{-1/4} \eps^{-1} M^{-1/4} |B|^{3/4}
\end{equation}
and thus
$$
    \sigma \lesssim \d^{-1/4} \eps^{-1} M^{-1/4} |B|^{3/4} \cdot \min\{ |C|,|D|\} + M^{3/4} |D|^{1/2} |C|^{3/4} |B|^{1/2} \,.
$$
Recall that $m:= \min\{ |C|,|D|\} = |C|$.
The optimal choice of $M$ is
$$
	M= \d^{-1/4} \eps^{-1} |B|^{1/4} |D|^{-1/2} |C|^{-3/4} m = \d^{-1/4} \eps^{-1} |B|^{1/4} |D|^{-1/2} |C|^{1/4}
$$
 and hence
$$
    \sigma \lesssim \d^{-3/16} \eps^{-3/4} m^{3/4} |B|^{11/16} |D|^{1/8} |C|^{3/16}
    	=
    		\d^{-3/16} \eps^{-3/4} |B|^{11/16} |D|^{1/8} |C|^{15/16} \,.
$$
It is easy to see that the inequality $M\ge 1$ is equivalent to
$$
	|B| |C| \ge \d \eps^4 |D|^2
$$
but in view of (\ref{tmp:trivial_bc}) it would follows from
$$
	|B|^4 \le \d^{-4} \eps^{-16} \,.
$$
	The last inequality is a simple consequence of (\ref{tmp:Par_B}).
    Now let us check condition (\ref{cond:later}).
    In view of estimate (\ref{f:P_j_size}) it is sufficient to have
$$
    |U_j| |P_j| |C| \lesssim |Y_j|^4 |U_j| |C| / \E^{+} (U_j) \le M |C| |B|^2
        =
            \d^{-1/4} \eps^{-1} |B|^{9/4} |D|^{-1/2} |C|^{5/4} \ll p^2 \,.
$$
    The last bound is our condition (\ref{cond:technical_mc_add})
     (again we ignore signs $\ll$, $\gg$ increasing the constants in the final inequalities as in the proof of Lemma  \ref{l:D_Fp_pre}).

	It remains to consider the case when the second term in (\ref{f:AMRS_1}) dominates.
    We will show that in this situation one has even better upper bound for $\sigma$.
    Put $\nu_j = \max\{ |U_j|, |P_j|, |C| \})$.
	In view of formulas (\ref{f:AMRS_1}), (\ref{f:split_1})---(\ref{f:split_1'}) and our choice of $q_j$,
	it is sufficient to check
	$$
		|D|^{1/2}  \sum_{j} q_j^{-1} (|U_j| |P_j| |C| \nu_j)^{1/2}
			\lesssim
				|D|^{1/2}  \sum_{j} t_j^{-1} |Y_j| (|U_j| |P_j|^{-1} |C| \nu_j)^{1/2}
					\lesssim M^{3/4} |D|^{1/2} |C|^{3/4} |B|^{1/2} \,.
	$$
If $\nu_j = |P_j|$, then we obtain the inequality to
insure
$$
	\sum_j |Y_j| t_j^{-1} |U_j|^{1/2}   \lesssim M^{3/4} |C|^{1/4} |B|^{1/2} \,.
$$
	Clearly, $|U_j| \le |B|$, $t_j \gtrsim |B|/M$ and $\sum_j |Y_j| \le |B|$. Thus we need to check
$$
	M \le |C| 
$$
or, in other words,
$$
	|B| \le |C|^3 |D|^2 \d \eps^4
$$
and this is the first part of condition (\ref{cond:technical_mc}).
If $\nu_j = U_j$, then  we have the bound
$$
	\sum_j |Y_j| |U_j| (t_j^2 |P_j|)^{-1/2} \lesssim  \sum_j (\E^{+} (B))^{-1/2} |Y_j| |U_j| \le (|B|^3/M)^{-1/2} \sum_j |Y_j| |U_j|
		\le
			M^{1/2} |B|^{1/2} \,.
$$
Clearly, the last quantity is less than $M^{3/4} |C|^{1/4} |B|^{1/2}$.
Finally, if $\nu_j = |C|$, then similarly, we get
$$
	\sum_j |Y_j| |U_j|^{1/2} |C|^{1/2} (t_j^2 |P_j|)^{-1/2}
			\lesssim
			M^{1/2} |C|^{1/2} \,.
$$
To
make this less than
$M^{3/4} |C|^{1/4} |B|^{1/2}$
it is sufficient to have
$$
	|C| \le M |B|^2
$$
or
$$
	|C|^3 |D|^2 \d \eps^4 \le |B|^9 \,.
$$
The last inequality coincides with the second part of conditions (\ref{cond:technical_mc}).
Thus, we have proved the second part of our proposition.

It remains to obtain (\ref{f:technical_mc3}).
The first bound is a trivial consequence of the Cauchy--Schwarz inequality, combining with Theorem  \ref{t:energy*spec}.
Here we simply ignore that the summation is taken over the set $C$.
Notice that we do not use condition (\ref{cond:technical_mc_add'}) as well as $|C| \le |D| < p^{2/3}$, $|B| < p^{2/3}$
to obtain this bound but the assumption  $|B| < \d^{-1/6} \eps^{-2/3} \sqrt{p}$
only.
Let us prove the second estimate, where we need all mentioned assumptions.
In our arguments we apply  the algorithm above and construct the sets $U,V$, $U\sqcup V=B$, in particular.
Using the Cauchy--Schwarz inequality, Lemma \ref{l:norm_E+} and bound (\ref{f:U_bound}), we get
$$
	\sum_{x\in C} (B \otimes D)^2 (x) =  \sum_{x\in C} ((U \otimes D) (x) + (V \otimes C) (x))^2
	\ll
		\sum_{x\in C} (U \otimes D)^2 (x) + \sum_{x\in C} (V \otimes D)^2 (x)
			 \le
$$
$$
    \le
            |U|^2 |C| +
		 			(\sum_{j=1}^k (\E^{\times}_C (Y_j,D))^{1/2} )^2
		 		\lesssim
		 			\d^{-1/2} \eps^{-2} M^{-1/2} |B|^{3/2} |C| + \sigma_* \,.
$$
Here $M$ is a parameter which we will choose later.
Our task is to find a good upper bound for $\sigma_*$.
To estimate the sum $\sigma_*$ we need to bound $\E^{\times}_C (Y_j,D)$
via Corollary \ref{c:D_Fp_pre} with $A=Y_j$, $X=C$, $Y=D$, $P=P_j$, $Q=U_j$ and $t=q_j$.
To apply this corollary we have to
find the condition on the parameter $M$ when
\begin{equation}\label{f:cond_P_Y}
    |P_j|^2 |U_j|^{5/4} \ge q_j^{2} |Y_j|^{3/4} \,.
\end{equation}
Suppose not.
Then by formula $q_j |Y_j| \sim t_j |P_j|$, we obtain
$$
    |P_j|^2 |U_j|^{5/4} \lesssim t_j^2 |Y_j|^{-5/4} |P_j|^2
$$
and
because of $q_j \lesssim |Y_j|$,
we have
$|Y_j|^2 \gtrsim |Y_j| q_j \sim |P_j| t_j$ and hence
$$
    |U_j|^{5/4} (|P_j| t_j)^{5/8} \lesssim |U_j|^{5/4} |Y_j|^{5/4} \lesssim t_j^2 \,.
$$
This implies
$$
    |U_j|^{10} |B|^{15} M^{-5} \lesssim |U_j|^{10} \E^{+} (U_j)^5 \lesssim t_j^{21} \le |U_j|^{21} \,.
$$
Thus (\ref{f:cond_P_Y}) takes place if
\begin{equation}\label{f:M_B}
    M\lesssim |B|^{4/5} \,.
\end{equation}
In this  case the conditions of Corollary \ref{c:D_Fp_pre} take place because $|D| <p^{2/3}$ and $|U_j| \le |B| < p^{2/3}$.
Applying this corollary,
formulas $q_j |Y_j| \sim t |P_j|$, $t_j^2 |P_j| \sim \E^{+} (U_j)$
and inequality (\ref{f:P_j_size}), we obtain
$$
	\sigma_* \le |D|^{3/2} |C|^{1/3} (\sum_{j=1}^k q_j^{-1} |P_j|^{2/3} |U_j|^{3/4} )^2
		\lesssim
				|D|^{3/2} |C|^{1/3} |B|^{3/2} (\sum_{j=1}^k |Y_j| t_j^{-1}  |P_j|^{-1/3})^2
			\lesssim
$$
$$
			\lesssim
				|D|^{3/2} |C|^{1/3} |B|^{3/2} \left(\sum_{j=1}^k |Y_j| |P_j|^{1/6} (\E^{+} (U_j))^{-1/2} \right)^2
					\lesssim
$$
$$
                    \lesssim
						|D|^{3/2} |C|^{1/3} |B|^{3/2} \left(\sum_{j=1}^k |Y_j|^{5/3} (\E^{+} (U_j))^{-2/3} \right)^2
							\le
								M^{4/3} |D|^{3/2} |C|^{1/3} |B|^{5/6} \,.
$$
Thus
$$
	\sum_{x\in C} (B \otimes D)^2 (x)
		\lesssim
		 			\d^{-1/2} \eps^{-2} M^{-1/2} |B|^{3/2} |C|
		 				+
		 					M^{4/3} |D|^{3/2} |C|^{1/3} |B|^{5/6} \,.
$$
The optimal choice of $M$ is
$$
	M = |B|^{4/11} |C|^{4/11} \d^{-3/11} \eps^{-12/11} |D|^{-9/11}
$$
and hence
$$
	\sum_{x\in C} (B \otimes D)^2 (x)
		\lesssim
			\d^{-4/11} \eps^{-16/11} |C|^{9/11} |B|^{29/22} |D|^{9/22}
$$
as required.
It remains to notice that the condition
$M \lesssim |B|^{4/5}$ is equivalent to (\ref{cond:technical_mc_add'}).
This completes the proof.
$\hfill\Box$
\end{proof}

\bigskip

Bound (\ref{f:technical_mc2}) works better than (\ref{f:technical_mc1}) or (\ref{f:technical_mc3})
in the case when
the size of $D$ is large comparable to $B$ and $C$.
For very small $D$ estimate (\ref{f:technical_mc1}) is the best one.


\bigskip

Let us
prove
our main result.

\begin{theorem}
    Let $\mathcal{E}$ be a finite family of equations of form (\ref{f:E_form}).
    Also, let $A \subseteq \F_p$ be a set avoiding the family $\mathcal{E}$ and $|A| \gg p^{\frac{39}{47}}$.
    Then for any
    $\kappa_1 < \frac{10}{31}$,
    one has
\begin{equation}\label{f:main1}
    |A| \ll
        \frac{p}{\mathcal{T}(\mathcal{E})^{\kappa_1}} \,.
\end{equation}
    and
    for an arbitrary
    $\kappa_2 < \frac{3}{10}$
    the following holds
\begin{equation}\label{f:main2}
    |A| \ll
            \frac{p}{\mathcal{T}_* (\mathcal{E})^{\kappa_2}} \,.
\end{equation}
    Finally, let $|A| \gg p^{7/9}$, $\mathcal{T}_* (\mathcal{E}) <p^{2/3}$.
    Then for an arbitrary $\kappa_3 < \frac{35}{159}$ one has
\begin{equation}\label{f:main3}
  |A| \ll \max\left\{ \frac{p}{\mathcal{T}_* (\mathcal{E})^{\kappa_3}} \cdot \left( \frac{\E^{+}_* (A)}{|A|^3} \right)^{\frac{22}{159}},
                 \frac{p}{\mathcal{T}_* (\mathcal{E})^{69/183}} \right\} \,.
\end{equation}
\label{t:main}
\end{theorem}
\begin{proof}
    Let $|A| = \d p$ and $t=\mathcal{T}(\mathcal{E})$.
    By assumption the set $A$ avoids all equations from the family $\mathcal{E}$.
    Using the Fourier transform, we see that it is equivalent to
\begin{equation}\label{tmp:23.10.2016_1}
    0 = \sum_r \FF{A} (a_j r) \FF{A} (b_j r) \FF{A} (c_j r) e(-d_j r)
        = |A|^3 + \sum_{r\neq 0} \FF{A} (a_j r) \FF{A} (b_j r) \FF{A} (c_j r) e(-d_j r)
\end{equation}
    for all $j \in [t]$.
    Applying Parseval identity (\ref{F_Par}) three times, we have
\begin{equation}\label{f:first_25.07}
    2^{-1} |A|^3
        \le
            \sum_{r \in (a^{-1}_j B) \cap (b^{-1}_j B) \cap (c^{-1}_j B) } |\FF{A} (a_j r)| |\FF{A} (b_j r)| |\FF{A} (c_j r)| \,,
\end{equation}
    where $B = \Spec_{\eps} (A) \setminus \{0 \}$, $\eps =\d/6$. Indeed,
$$
    \sum_{r \notin c^{-1}_j B } |\FF{A} (a_j r)| |\FF{A} (b_j r)| |\FF{A} (c_j r)|
        \le
            \eps |A| \sum_{r} |\FF{A} (a_j r)| |\FF{A} (b_j r)|
                \le
$$
$$
                \le
                    \eps |A|
                        \left( \sum_{r} |\FF{A} (a_j r)|^2 \right)^{1/2}
                        \left( \sum_{r} |\FF{A} (b_j r)|^2 \right)^{1/2}
                            \le
                                \eps |A| |A| p = |A|^3 / 6
$$
and similar for another two terms.
Here we have used that $a_j,b_j,c_j$ are nonzero numbers.

Let
$S:= \{ s_1,\dots, s_t\} \subseteq S(\mathcal{E})$ such that, say, $s_j = (a_j,b_j,1)$, where $a_j,b_j$ are different.
In particular, $a_j$, $b_j$ belong to two sets $S_A$, $S_B$, correspondingly,
and $|S_A| = |S_B| =t$.
Now let us return to (\ref{f:first_25.07}).
Summing the last
estimate
over $S$, and using the Cauchy--Schwartz inequality and Parseval identity (\ref{F_Par}), we have
$$
	t^2 |A|^6
		\ll
            \left( \sum_{r\in B} \, |\FF{A} (r)| \sum_{j=1}^t |\FF{A} (a_j r)| |\FF{A} (b_j r)| B(a_j r) B(b_j r) \right)^2
        \ll
$$
$$
        \ll
			\sum_{r} |\FF{A} (r)|^2  \cdot
				\sum_{r \in B} \left( \sum_{j=1}^t\,  |\FF{A}(a_j r)| |\FF{A}(b_j r)| B(a_j r) B(b_j r) \right)^2
    \le
$$
$$
    \le
        |A| p  \cdot \sum_{r \in B} \left( \sum_{j=1}^t\,  |\FF{A}(a_j r)| |\FF{A}(b_j r)| B(a_j r) B(b_j r) \right)^2
    \,.
$$
Using the pigeonholing  principle twice, we find two numbers $\Delta_1$, $\Delta_2 \le |A|$ and two sets $W_1,W_2 \subseteq B$ such that
 $\D_1 < |\FF{A} (r)|\le 2 \D_1$ for $r\in W_1$, $\D_2 < |\FF{A} (r)|\le 2 \D_2$ for $r\in W_2$ and
$$
	t^2 |A|^6 \lesssim
            |A| p \D_1^2 \D_2^2 \cdot \sum_{r\in B} \left( \sum_{j=1}^t W_1 (a_j r) W_2 (b_j r) \right)^2
    \le
$$
\begin{equation}\label{f:new_basic}
    \le
|A| p \D_1^2 \D_2^2 \cdot \sum_{r\in B} (S^{-1}_A \otimes W_1) (r) (S^{-1}_B \otimes W_2) (r) \,.
\end{equation}
Put $\eps_1 = \D_1/|A|$, $\eps_2 = \D_2/|A|$.
In particular, from formula (\ref{f:new_basic}), combining with (\ref{F_Par}), we get
\begin{equation}\label{tmp:24.08_0}
    t^2 |A|^6 \lesssim |A| p \D_1^2 \D_2^2 |W_1| |W_2| |B| \le (|A| p)^3 |B|
\end{equation}
and hence
\begin{equation}\label{tmp:24.08_1}
    |B| \gtrsim \d^3 t^2 \,.
\end{equation}
as well as
\begin{equation}\label{tmp:24.08_2}
    |B| |W_1| \gtrsim \d^2 t^2  \quad \mbox{ and } \quad
    |B| |W_2| \gtrsim \d^2 t^2 \,.
\end{equation}
In particular, in view of $|B| \ll \d^{-3}$, we obtain
\begin{equation}\label{tmp:24.08_10}
    |W_1|, |W_2| \gtrsim \d^5 t^2 \,.
\end{equation}
The singes $\lesssim$, $\gtrsim$ in formulas  (\ref{tmp:24.08_0})---(\ref{tmp:24.08_10}) as well as in all formulas below depend on the size of the set $|B|$.
The last quantity is less than  $O(\d^{-3})$ and so it depends on the density of the set $A$ but on the size of $A$.
Thus  we can remove these logarithms requiring strictly smaller power of $\d$
in the formulation of the theorem.

Now put $m=\min \{ |W_1|, |W_2| \}$ and let $m=|W_2|$ for certainty.
Further we have
$$
    t^2 |A|^6 \lesssim |A| p \D_1^2 \D_2^2 m \cdot \sum_{r\in B} (S^{-1}_A \otimes W_1) (r)  \,.
$$
As above by (\ref{f:spec_Par}) we see that $|B| \ll \d^{-3}$, so if $\d^{-3} \gg t$, then $\d \ll t^{-1/3}$ and it is nothing to prove.
Hence one can assume that $|W_1|, |W_2|, |B| \le t$ (again we ignore signs $\ll$, $\gg$ increasing the constants in the final inequalities
as in Lemma  \ref{l:D_Fp_pre}).
Split the set $B$ into some $s$ sets $B_j$ of approximately equal sizes, where $s$ is a parameter which we will choose later.
Using the  bound $|B| \ll \d^{-3}$ and applying the Parseval identity one more time as well as  the second part of Proposition \ref{p:technical_mc}
with $\eps = \eps_1 = \D_1/|A|$, $B =W_1$, $C=B^{-1}_j$, $D=S_A$, we obtain
\begin{equation}\label{f:calc_1}
    t^2 \d^6 p^6 =
        t^2 |A|^6 \lesssim  |A| p \D_1^2 \D_2^2 m \sum_{j=1}^s \sum_{r\in B_j} (S^{-1}_A \otimes W_1) (r)
            \lesssim
\end{equation}
\begin{equation}\label{f:calc_1.5}
            \lesssim
                |A| p \D_1^2 \D_2^2 m \d^{-3/16} (\D_1 / |A|)^{-3/4} t^{1/8} |W_1|^{11/16} \sum_{j=1}^s |B_j|^{15/16}
        \ll
\end{equation}
\begin{equation}\label{f:calc_2-}
        \ll
            |A|^2 p^2 \d^{-3} |A|^{3/4} \D_1^{5/4} |W_1|^{11/16} t^{1/8} s^{1/16}
                \le
                    \d^{-1/4} p^{19/4} (\D^2_1 |W_1|)^{5/8} |W_1|^{1/16} t^{1/8} s^{1/16}
                    \le
\end{equation}
\begin{equation}\label{f:calc_2}
                    \le
                        p^6 \d^{3/8} |W_1|^{1/16} t^{1/8} s^{1/16} \,.
\end{equation}
It remains to check that all conditions of Proposition \ref{p:technical_mc} satisfy
and choose the parameter $s$.
We have already insured  that $|C| \le |D|$ since $|W_1|, |W_2|, |B| \le t$.
If
\begin{equation}\label{f:B_j->B}
        p^2 <\d^{-1/4} \eps^{-1}_1 |W_1|^{9/4}   t^{-1/2} |B_j|^{5/4}
            \le
                \d^{-1/4} \eps^{-1}_1 |W_1|^{9/4}   t^{-1/2} |B|^{5/4}
                    \ll
\end{equation}
$$
                    \ll
                        \d^{-1/4-1-3 \cdot 7/2} = \d^{-47/4}
            \,,
$$
then one can easily arrives  to a contradiction with the assumption $|A| \gg p^{39/47}$.
Thus condition (\ref{cond:technical_mc_add}) takes place.
Further from (\ref{tmp:24.08_1}) for any $j$, it follows that
$$
    \d \eps^4_1 |B_j|^3 t^2 \gg \d^5 |B|^3 t^2 s^{-3} \ge \d^5 |W_1| |B|^2 t^2 s^{-3} \gtrsim \d^{11} |W_1| t^6 s^{-3} \ge |W_1|
$$
provided $s \le \d^{11/3} t^2$.
Further
$$
    \d \eps_1^4 |B_j|^{3} t^2 \ll \d |B|^{3} t^2 s^{-3} \ll
        \d^{-8} t^{2} s^{-3}
            \le |W_1|^9
$$
provided $s\gg \d^{-8/3} t^{2/3} |W_1|^{-3}$.
Putting $s= \d^{-8/3} t^{2/3} |W_1|^{-3}$ one can insure that $s \ll \d^{11/3} t^2$
because otherwise in view of (\ref{tmp:24.08_10}), we have
$$
    \d^{-8} t^{2} \gtrsim \d^{11} t^6 |W_1|^9 \gtrsim \d^{56} t^{24}
$$
or, in other words,
$\d\lesssim t^{-11/32}$ which is better than (\ref{f:main1}).
If $s\ll 1$, then from (\ref{f:calc_2}), we obtain
$$
    t^{15/8} \d^{45/8} \lesssim |W_1|^{1/16} \ll \d^{-3/16}
$$
or
$$
    \d \lesssim t^{-10/31} \,.
$$
The last bound coincides with (\ref{f:main1}).
Finally, we should note that in the  case $s\ll 1$ one quickly insure that the condition
$
    \d \eps_1^4 |B|^{3} t^2 \gtrsim |W_1|
$
takes place.
Thus from (\ref{f:calc_2}), (\ref{tmp:24.08_10}) and our choice of the parameter $s$, it follows that
$$
    t^2 \d^6 \lesssim \d^{3/8} t^{1/8} \d^{-1/6} t^{1/24} |W_1|^{-1/8}
        \lesssim
            \d^{3/8} t^{1/8} \d^{-1/6} t^{1/24} \cdot \d^{-5/8} t^{-1/4}
$$
or
$$
    \d \lesssim t^{-25/77}
$$
which is better than (\ref{f:main1}) again.

\bigskip

Now let us prove the second part of the theorem.
Put $t_* = \mathcal{T}_* (\mathcal{E})$ and let $S_* = \{ s_1,\dots, s_{t_*}\}$ be the set from the Definition \ref{def:T_*}
(without loosing of the generality we consider the intersection of $S(\mathcal{E})$ with the plane $\{ z=1\}$).
Returning to (\ref{f:first_25.07}) and then after changes of variables, we have for any $j\in [t_*]$ that
$$
    2^{-1} |A|^3 \le \sum_{r \in B \cap (B/s_j) \cap (B/s'_j)} |\FF{A} (s_j r)| |\FF{A} (s'_j r)| |\FF{A} (r)| \,.
$$
Here $s'_j = a_j$ if $s_j = b_j$, further $s'_j = b_j$ if $s_j = a_j$ and, finally,  $s'_j = b^{-1}_j$ if $s_j = a_j/b_j$.
Since $s'_j \neq 0$, we obtain by the Cauchy--Schwarz inequality and formula (\ref{F_Par})
\begin{equation}\label{tmp:25.08_2}
    |A|^6 \ll |A|p \sum_{r \in B \cap (B/s_j)} |\FF{A} (s_j r)|^2 |\FF{A} (r)|^2 \,.
\end{equation}
Thus, summing over $j\in [t_*]$, we get
\begin{equation}\label{tmp:26.08_2}
    |A|^5 t_* \ll p \sum_{j=1}^{t_*}\, \sum_{r \in B \cap (B/s_j)} |\FF{A} (s_j r)|^2 |\FF{A} (r)|^2 \,.
\end{equation}
Using the pigeonholing  principle twice, we find two numbers $\Delta_1$, $\Delta_2 \le |A|$ and two sets $W_1,W_2 \subseteq B$ such that
 $\D_1 < |\FF{A} (r)|\le 2 \D_1$ for $r\in W_1$, $\D_2 < |\FF{A} (r)|\le 2 \D_2$ for $r\in W_2$ and
\begin{equation}\label{tmp:30.07_1'}
    |A|^5 t_* \lesssim p \D_1^2 \D_2^2  \sum_r W_2 (r) (W_1 \otimes S^{-1}_*) (r) \,.
\end{equation}
As above, we have $|W_1|, |W_2| \le t_*$ because otherwise it is nothing to prove.
Similarly, one can check that the conditions
$$
    \d^{-1/4} \eps_1^{-1} |W_1|^{9/4}   t^{-1/2}_* |W_2|^{5/4} < p^2 \,, \quad \quad
    \d^{-1/4} \eps_2^{-1} |W_2|^{9/4}   t^{-1/2}_* |W_1|^{5/4} < p^2 \,.
$$
follows from the assumption $|A| \gg p^{39/47}$.
Here $\eps_1 = \D_1/|A|$ and $\eps_2 = \D_2 / |A|$.
Further from (\ref{tmp:30.07_1'}), we obtain
$$
	t_* \d |A|^4 \ll \D_1^2 \D_2^2 |W_1| |W_2| \,.
$$
Whence in view of the Parseval identity, we have
\begin{equation}\label{tmp:30.07_2}
	|W_1|, |W_2| \gg \d^2 t_* \quad \quad \mbox{ and } \quad \quad |W_1| |W_2| \gg \d t_* \,.
\end{equation}
Suppose that $|W_2| \ge |W_1|$ for certainty.
Split the set $W_2$ into some $s$ sets $W^{(j)}_2$ of approximately equal sizes, where $s$ is a parameter which we will choose later.
Bounds (\ref{tmp:30.07_2})
imply for any $j$
$$
    \d \eps^4_1 |W^{(j)}_2|^3 t_*^2
        \gg
        \d^5 |W_2|^3 t_*^2 s^{-3}
            \gg
            |W_1|
$$
provided $s\ll |W_2| t_*^{2/3} \d^{5/3} |W_1|^{-1/3}$.
Further
$$
    \d \eps^4_1 |W^{(j)}_2|^3 t_*^2  \ll \d |W_2|^3 t_*^2 s^{-3}
    \ll |W_1|^9
$$
provided $s\gg \d^{1/3} |W_2| t_*^{2/3} |W_1|^{-3}$.
Putting $s= \d^{1/3} |W_2| t_*^{2/3} |W_1|^{-3}$ one can insure that
$s\ll |W_2| t_*^{2/3} \d^{5/3} |W_1|^{-1/3}$ because otherwise in view of (\ref{tmp:30.07_2}), we have
$$
    1 \gg |W_1|^{8/3} \d^{4/3} \gg t_*^{8/3} \d^{20/3}
$$
or, in other words, $\d \lesssim t_*^{-2/5}$ which is better than (\ref{f:main2}).
Suppose, in addition, that $s\gg 1$.
Thus all conditions of the second part of Proposition \ref{p:technical_mc} takes place.
Applying arguments as in (\ref{f:calc_1})--(\ref{f:calc_2}), bounds (\ref{tmp:30.07_2})
and using
Proposition \ref{p:technical_mc} with $\eps=\eps_1$, $B=W_1$, $C= (W^{(j)}_2)^{-1}, D=S_*$ and the Parseval identity, we have
\begin{equation}\label{f:calc_1'}
    |A|^5 t^{7/8}_* \lesssim p \D_1^2 \D_2^2  \d^{-3/16} (|A|/\D_1)^{3/4} |W_1|^{11/16} (|W_2|/s)^{15/16} s
        =
\end{equation}
$$
    =
            p |A|^{3/4} \D_1^{5/4} \D_2^2 \d^{-3/16} |W_1|^{11/16} |W_2|^{15/16} s^{1/16}
                \le
$$
\begin{equation}\label{f:calc_2'}
                \le
                    p |A|^{3/4} (p|A|)^{13/8} \d^{-3/16} |W_1|^{1/16} |W_2|^{-1/16} s^{1/16}
                        =
                            p^{21/8} |A|^{19/8} \d^{-3/16} |W_1|^{1/16} |W_2|^{-1/16}  s^{1/16}
\end{equation}
$$
    \ll
         p^{5} \d^{35/16} (\d^{1/3} t_*^{2/3} |W_1|^{-2})^{1/16}
            \ll
                p^{5} \d^{35/16} (\d^{-11/3} t_*^{-4/3} )^{1/16} \,.
$$
It gives us
\begin{equation}\label{f:first_est_dt'}
    \d \lesssim t_*^{-23/73}
\end{equation}
which is better than (\ref{f:main2}).
If $s \ll 1$ then from (\ref{f:calc_2'}) and (\ref{tmp:30.07_2}), we see that
$$
    \d^5 t^{7/8}_* \lesssim \d^{35/16} |W_1|^{1/16} |W_2|^{-1/16} \lesssim \d^2  (\d^2 t_*)^{-1/16}
        =
        \d^{15/8} t_*^{-1/16}
$$
or, in other words,
$$
    \d \lesssim t_*^{3/10}
$$
which coincides with (\ref{f:main2}).
Finally, we should note that in the  case $s\ll 1$
in view of the inequality $|W_2| \ge |W_1|$ and bound (\ref{tmp:30.07_2}), we easily have
$$
    \d \eps^4_1 |W_2|^3 t_*^2
        \gg
        \d^5 |W_1| |W_2|^2 t_*^2
            \gg
                \d^9 |W_1| t_*^4
                 \gg
            |W_1|
$$
because otherwise we obtain $\d \lesssim t_*^{-4/9}$ which is much better than (\ref{f:main2}).

\bigskip

It remains to prove the last part of the theorem.
Returning to (\ref{tmp:26.08_2}) and squaring, we obtain
$$
    |A|^{10} t_*^2 \lesssim p^3 \E^{+}_* (A) \D^4 \sum_{r\in B} (W \otimes S^{-1}_*)^2 (r) \,.
$$
Here $\D := \eps |A| \le |A|$ and $W \subseteq B$ comes from the pigeonhole principle as above.
Notice that the condition $|A| \gg p^{7/9}$ implies
$$
    |B| \le (\d \eps^2)^{-1} \ll \d^{-3} < p^{2/3} \,.
$$
Further the condition (recall the inequality $|W| \gtrsim \d^2 t_*$)
$$
    t_*^{45} \d^{15} \eps^{60} |W|^{24}
        \gg
        t_*^{45} \d^{75} |W|^{24}
            \gtrsim
                t_*^{69} \d^{123}
            \gg
            |B|^{20}
$$
trivially holds  because otherwise
$$
    \d \ll t_*^{-69/183} \,.
$$
Hence, applying the third part of Proposition \ref{p:technical_mc} with $C=B$, $B=W$, $D=S_*^{-1}$, $\eps =\D/|A|$
    and the Parseval identity, we get
$$
    |A|^{10} t_*^{35/22}
        \lesssim
            p^3 \E^{+}_* (A) \D^4 \d^{-4/11} (|A|/\D)^{16/11} |B|^{9/11} |W|^{29/22}
                =
$$
$$
                =
                    p^3 \E^{+}_* (A) |A|^{16/11} \D^{28/11} \d^{-4/11} |B|^{9/11} |W|^{29/22}
                        \le
                            p^3 \E^{+}_* (A) |A|^{16/11} (|A|p)^{28/22} \d^{-4/11} |B|^{9/11} |W|^{1/22} \,.
$$
Using $|B|, |W| \ll \d^{-3}$, we have
$$
    \d^{159/22} \lesssim \E^{+}_* (A)/|A|^3 \cdot t_*^{-35/22}
$$
or
$$
    \d \lesssim t_*^{-35/159} \cdot \left( \E^{+}_* (A)/|A|^3 \right)^{22/159} \,.
$$
This completes the proof.
$\hfill\Box$
\end{proof}

\bigskip

In view of Lemma \ref{l:TE},
we obtain

\begin{corollary}
    Let $\mathcal{E}$ be a finite family of equations of form (\ref{f:E_form}).
    Also, let $A \subseteq \F_p$ be a set avoiding the family $\mathcal{E}$, $|A| \gg p^{\frac{39}{47}}$.
    Then for any $\kappa < \frac{5}{31}$ one has
$$
    |A| \ll \frac{p}{|\mathcal{E}|^{\kappa}} \,.
$$
\label{c:result_with_E}
\end{corollary}

\bigskip

If one use the parameter $s$ in estimate (\ref{f:B_j->B}), then the restriction $|A| \gg p^{\frac{39}{47}}$ in the first two parts of Theorem \ref{t:main} as well as in  Corollary \ref{c:result_with_E} can be refined.
We do not make such calculations.

\bigskip

Clearly, Proposition \ref{p:low_bounds}, combining with Theorem \ref{t:main} give Theorem \ref{t:main_intr} from the introduction.
Further it is easy to see that $\E^{+} (A) = o(|A|^3)$ implies that any set avoiding just {\it one} equation has size $o(p)$.
Inequality (\ref{f:main3}) can be considered as a generalization of this fact for several equations.

\begin{remark}
    It is easy to see from the proof of Theorem \ref{t:main} that the same arguments work for sets having, say,
    at most $|A|^3/(4p)$ or at least $2|A|^3/p$ solutions of equations (\ref{f:E_form}).
    In other words, the number of solutions must differ from the expectation significantly.
\label{r:expectation}
\end{remark}

After this paper was written Tomasz Schoen found a simpler proof of the first part of Theorem \ref{t:main} without using sum--product method.
Indeed, let us make first steps (\ref{tmp:23.10.2016_1})---(\ref{tmp:24.08_1}).
Recalling that $|B| \ll \d^{-3}$, we get $\d \lesssim t^{-1/3}$.
More precisely,
we obtain
$$
    t^2 |A|^6 \ll |A| p \cdot \sum_{r\in B} \left( \sum_{j=1}^t |\FF{A} (a_j r)| |\FF{A} (b_j r)| \right)^2
        \le
            (|A| p)^3 |B| \ll (|A| p)^3 \d^{-3} \,,
$$
whence $\d \ll t^{-1/3}$.
So, in particular, this stronger result takes place in a general field.
We leave the old proof in the paper  because it is more effective in another regimes, see e.g. the proofs of Theorems \ref{t:T_k_new}, \ref{t:E_lx_new}.

\section{Further applications}
\label{sec:applications}


This section contains three applications of the results above.
Let us consider the first one.


In \cite{Ruzsa_non-av}, \cite{Non-av} authors
studied
a family of subsets of $\Z$ which generalize arithmetic progressions of length three.
Let us recall the definition.
Let $t\ge 1$ be  a fixed integer. A finite set $A\subset \Z$ is called {\it non--averaging of order $t$}, if for every $1\le m,n\le t$
the equation
\begin{equation}\label{def:non-av}
	m X_1 + n X_2 = (m+n) X_3
\end{equation}
have just trivial solutions: $X_1=X_2=X_3$.
For example, if $t=1$, then $A$ is non--averaging of order $1$ iff $A$ has no arithmetic progressions of length three.
The best upper bound for the size of a subset of $[N]$ having no arithmetic progressions of length three as well the history of the question
can be found in \cite{Bloom}.
Namely, developing the method of Sanders \cite{Sanders_AP3}, T.F. Bloom proved that
\begin{equation}\label{f:AP3}
	|A| \ll \frac{N (\log \log N)^{4}}{\log N} \,.
\end{equation}

Here we obtain a new upper bound for the size of a non--averaging set of order $t$ in $\F_p$, that is, a set having no non--trivial solutions of  system (\ref{def:non-av}) in $\F_p$.
It is known that the modular version of the question about the density of arithmetic progressions is equivalent to the integer case.
In particular, inequality (\ref{f:AP3}) takes place with $N=p$ for sets $A \subseteq \F_p$ without solutions $x+y \equiv  2z \pmod p$.

\begin{theorem}
	Let $A\subseteq \F_p$ be a non--averaging set of order $t$, $t< \sqrt{p}$.
    Then
\begin{equation}\label{f:non-av}
	|A| \ll \frac{p}{t^{2/3}} \,.
\end{equation}
\label{t:non-av}
\end{theorem}
\begin{proof}
By our assumption  the set $A$ avoids all equations from (\ref{def:non-av}).
In other words, $X_1 + n/m \cdot X_2 = (1+n/m) X_3$, where $X_1,X_2,X_3\in A$ implies $X_1=X_2=X_3$.
Thus, we have the correspondent system $\mathcal{E}$ with the set $S(\mathcal{E})$ of cardinality $|[t]/ [t]|$.
In a similar way
$$
	\mathcal{T} (\mathcal{E}) = |\{ n/m ~:~ n,m \in [t]\}| = |[t]/ [t]| \,.
$$
	Considering square--free numbers, it is easy to see in view of the assumption $t< \sqrt{p}$ that  $|[t]/ [t]| \gg t^2$
	both in $\Z$ and in $\F_p$ (consult the proof of Proposition \ref{p:low_bounds}).
	Although Theorem \ref{t:main} was formulated just for sets having no solutions at all,
    it is easy to insure that
the number of trivial solutions is $|A|$.
	Thus
if $|A| p \le |A|^3/4$, say,  then the method of the proof works (see Remark \ref{r:expectation}).
	Of course, if $|A| p > |A|^3/4$, then $|A| < 2\sqrt{p}$ and there is nothing to prove.
	Whence,
	applying the first  part of Theorem \ref{t:main} (and the arguments after Remark \ref{r:expectation}), we obtain the required result.
$\hfill\Box$
\end{proof}

\bigskip

Thus, taking any
$C>3/2$
and $t$ such that $t \ge (\log p)^{C}$, we see that bound (\ref{f:non-av}) is better than (\ref{f:AP3}) in this case.

\bigskip

Now consider another application.


Let $A\times A$ is the Cartesian product of a set $A\subseteq \F_p$.
The number of {\it collinear triples} $\T(A)$ in  $A\times A$ is an important characteristic of a set,
see \cite{AMRS}, \cite{s_multD}, \cite{S_Z}, say.
Observe
(or see \cite{s_multD}, \cite{S_Z})
that
$$
	\T(A) = \left| \left\{ \frac{a_1-a}{a'_1-a'} = \frac{a_2-a}{a'_2-a'}
					~:~ a_1,a_2,a,a'_1,a'_2,a' \in A \right\} \right| \,.
$$
(we suppose in the formula above that $a'_1=a'$ implies $a'_2=a'$ and vice versa).
Another formula for $\T(A)$ is (see \cite{s_multD}, \cite{S_Z} again)
\begin{equation}\label{f:T_1_f}
	\T(A) = \sum_{a,a'\in A} \E^\times(A-a, A-a') + O(|A|^4) \,.
\end{equation}
The quantity $\T[A]$ is
naturally
connected with the set
\begin{equation}\label{f:T_1_f'}
	R[A] := \left\{ \frac{a_1-a}{a_2-a} ~:~ a_1,a_2,a\in A,\, a_2 \neq a \right\} \,.
\end{equation}
Namely,
\begin{equation}\label{f:T_1_f''}
	\T(A) = \sum_{\lambda \in R[A]} q^2 (\lambda) + O(|A|^4) \,,
\end{equation}
where
$$
	q(\lambda) = \left| \left\{ \frac{a_1-a}{a_2-a} = \lambda ~:~ a_1,a_2,a\in A,\, a_2 \neq a \right\} \right| \,.
$$

\bigskip

In \cite{AMRS} authors obtained an upper bound for $\T(A)$ in the case of small sets $A$.

\begin{theorem}
	Let $A\subset \F_p$ be a set with $|A|<p^{2/3}$.
	Then
$$
	\T(A) \ll |A|^{9/2} \,.
$$
\end{theorem}

Now we extend this result to larger sets, obtaining
an asymptotic formula
for the quantity $\T(A)$.

\begin{theorem}
	Let $A\subseteq \F_p$ be a set, $|A| > p^{2/5}$.
	Then for some absolute constant $C>0$ the following holds
\begin{equation}\label{f:T_k_new}
	\left| \T (A) - \frac{|A|^6}{p} \right| \ll
        \log^C (p/|A|) \cdot
    |A|^{40/9} p^{2/9} \,.
\end{equation}
\label{t:T_k_new}
\end{theorem}
\begin{proof}
Put $|A| = a = \d p$, $q_* (\lambda) = q(\lambda) - a^3/p$.
Because of $\sum_\lambda q(\lambda) = a^2(a-1)$, we have
\begin{equation}\label{f:q'_av}
	\sum_\lambda |q_* (\lambda)| \le \sum_\lambda (q(\lambda) + a^3/p) \le 2 a^3 \,,
\end{equation}
$$
    \sum_\lambda q_* (\lambda) = - a^2 \,,
$$
and hence
\begin{equation}\label{tmp:08.08_1}
	\sum_\lambda q^2(\lambda) = \sum_\lambda \left( q_*(\lambda) + \frac{a^3}{p} \right)^2
	= \sum_\lambda q_*^2(\lambda) + \frac{a^6}{p} + \frac{2a^3}{p} \sum_\lambda q_* (\lambda)
		\le
			 \sum_\lambda q_*^2(\lambda) + \frac{a^6}{p} \,.
\end{equation}
Now for any $\tau \ge 1$ consider the set
$$
	S_\tau := \{ \lambda \neq 0,1 ~:~ |q_* (\lambda)| \ge \tau \} \,.
$$
Clearly, for an arbitrary  $\lambda \neq 0$ the number $q(\lambda)$ is
$$
	q(\lambda) := |\{ a_1,a_2,a\in A ~:~ a_1-\lambda a_2 + (\lambda-1)a = 0\,,\, a_2 \neq a \}|
		=
$$
$$
		=
			|\{ a_1,a_2,a\in A ~:~ a_1-\lambda a_2 + (\lambda-1)a = 0 \}| + 1 \,.
$$
Applying the Fourier transform, we get
\begin{equation}\label{f:q_*_Fourier}
	q_* (\lambda) = p^{-1} \sum_{r\neq 0} \FF{A} (r) \FF{A}(-\lambda r) \FF{A}( (\lambda-1) r) + 1 \,.
\end{equation}
	In particular, at least  $q_* (\lambda)/2$ of the mass of $q_* (\lambda)$ is contained in the set of non--zero
	$r$, $r\in \Spec_\eps (A) \cap \lambda^{-1} \Spec_\eps (A) \cap (\lambda-1)^{-1} \Spec_\eps (A)$,
	where $\eps = q_* (\lambda)/ (8|A|^2)$.
	Thus we have obtained $|S_\tau|$ linear equations of the form (\ref{f:E_form}).
    Also, it is easy to see that we have for the correspondent system $\mathcal{E}$ that $\mathcal{T} (\mathcal{E}) = |S_\tau|$.
Using the arguments and the notations of the proof of the second part of Theorem \ref{t:main},
we constructing the sets $W_1,W_2$ and the numbers $\D_1,\D_2$ such that
$$
    |S_\tau| \tau^2 p^2 \lesssim |A| p \D_1^2 \D_2^2 \sum_{r} W_2 (r) (W_1 \otimes S^{-1}_\tau) (r) \,.
$$
Applying the first part of Proposition \ref{p:technical_mc} with $B=W_1$, $C=W^{-1}_2$ and $D=S_\tau$
as well as the Parseval identity, we obtain
$$
    |S_\tau| \tau^2 p^2 \lesssim a p \D_1^2 \D_2^2 \d^{-1/3} (a/\D_1)^{2/3} (a/\D_2)^{2/3} |W_1|^{1/2} |W_2|^{1/2} |S_\tau|^{1/2}
        =
$$
$$
        =
            a^{7/3} p \D^{4/3}_1 \D^{4/3}_2 \d^{-1/3}   |W_1|^{1/2} |W_2|^{1/2} |S_\tau|^{1/2}
                \le
                    a^{10/3} p^2 \D_1^{1/3} \D_2^{1/3}  \d^{-1/3} |S_\tau|^{1/2}
                        \le
$$
\begin{equation}\label{tmp:12.09.2016_2}
    \le
    a^{10/3} p^2 (\| \FF{A} \|'_\infty)^{2/3}  \d^{-1/3} |S_\tau|^{1/2}
                        \le
                            \d^{11/3} p^6 |S_\tau|^{1/2}
\end{equation}
or, in other words,
\begin{equation}\label{f:S_tau_str}
    |S_\tau| \lesssim p^8 \d^{22/3} \tau^{-4}
\end{equation}
	provided
    the following conditions hold
\begin{equation}\label{tmp:08.08_c1}
	|W_1| < \d^{-1/6} (a/\D_1)^{2/3} \sqrt{p} \,, \quad \quad |W_2| < \d^{-1/6} (a/\D_2)^{2/3} \sqrt{p} \,.
\end{equation}
Let us check conditions (\ref{tmp:08.08_c1}) later.
In view of inequality (\ref{f:q'_av}), it follows that
$$
	\sum_{\lambda\neq 0,1} q_*^2(\lambda) \lesssim p^8 \d^{22/3} \tau^{-2} + \tau \sum_{\lambda} |q_* (\lambda)|
		\le
		   p^8 \d^{22/3} \tau^{-2}  + 2 \tau \d^3 p^3 \,.
$$
	The optimal choice of $\tau$ is $\tau = \tau_0 \sim \d^{13/9} p^{5/3} = a^{13/9} p^{2/9}$.
    Thus
$$
	\sum_{\lambda\neq 0,1} q_*^2(\lambda)
		\lesssim
			\tau_0 a^3
				\ll
					a^{40/9} p^{2/9} \,.
$$
	Returning to  (\ref{tmp:08.08_1}) and using
	$$
		-a \le q_* (0) = q_* (1) = a^2 -a - \frac{a^3}{p} \le a^2 \,,
	$$
	we obtain
$$
	|\T (A)  - \frac{a^6}{p}|   \lesssim   a^4 + a^{40/9} p^{2/9} \lesssim a^{40/9} p^{2/9}
$$
as required.

It remains to insure that  conditions  (\ref{tmp:08.08_c1}) takes place
and it is sufficient to check them  for $\tau \ge \tau_0$.
Put $\eps_1 = \D_1 / |A|$, $\eps_2 = \D_2 / |A|$.
Further it is easy to see that $\eps_1, \eps_2 \gg \tau / a^2 \ge \tau_0/ a^2$
and hence  (\ref{tmp:08.08_c1})
is a consequence of
the Parseval identity and the following estimates
$$
    |W_1| \le p/ (a \eps^2_{1}) < \d^{-1/6} \eps_{1}^{-2/3} \sqrt{p}
$$
or, in other words,
\begin{equation}\label{tmp:12.09.2016_1}
    \d^{1/6} \sqrt{p} = a^{1/6} p^{1/3} \ll a (a^{-5/9} p^{2/9})^{4/3} = a (\tau_0 /a^2)^{4/3} < a \eps_{1}^{4/3} \,.
\end{equation}
The first inequality in (\ref{tmp:12.09.2016_1}) follows from the condition $a> p^{2/5}$.
Similar bound takes place for the set $W_2$.
This completes the proof.
$\hfill\Box$
\end{proof}


\begin{remark}
    Let $p$ be a prime number and let  $A$ be a subfield of $\F_p^2$ of order $p$.
    Then $\T(A) \gg |A|^5 = p^5$ but bound (\ref{f:T_k_new}) (if it would holds in general fields) gives us $\T(A) \lesssim p^{40/9} \cdot (p^2)^{2/9} = p^{5-1/9}$.
    Thus, we need in the condition that $p$ is a prime number in the result above.
\end{remark}


Of course, estimate (\ref{f:T_k_new}) is an asymptotic formula just for sets $A$ with $|A| \gtrsim  p^{11/14}$.
For sets $A$, having the medium size $p^{2/5} < |A| \lesssim p^{11/14}$ inequality (\ref{f:T_k_new}) is just a non--trivial  upper bound for the quantity $\T(A)$.
Also, notice that one can improve bound (\ref{f:T_k_new}), using knowledge about $\| \FF{A}\|'_\infty$, see estimate (\ref{tmp:12.09.2016_2}).

\bigskip

Using formulas (\ref{f:T_1_f}), (\ref{f:T_1_f'}), (\ref{f:T_1_f''}) and the Cauchy--Schwarz inequality, we obtain

\begin{corollary}
    Suppose  $A \subseteq \F_p$ such that $|A| \gtrsim  p^{11/14}$.
	Then
$$
	|R[A]|
        \ge
            (1-o(1)) p
    \,.
$$
\end{corollary}



\bigskip

The last application of this section concerns mixed energies of a set.

In \cite{RSS}, see  Lemma 21, developing the investigations from \cite{Petridis}
(see Theorem 2 from here),
authors obtained a sum--product result for sets $A$, $p^{1/2}<|A|\leq p^{2/3}$, namely

\begin{lemma}\label{inccount}
Let $A\subseteq \mathbb{F}_p,$ $X\subseteq \mathbb{F}_p^*$.
Suppose
$|X|=O(|A|^2)$ and $|A|^2|X|=O(p^2)$. Then
$$\sum_{x\in X}\E^+(A,xA) \ll \E^+(A)^{1/2}       |A|^{3/2}|X|^{3/4} \,.$$
\end{lemma}

It is easy to see that the method of the proof of Theorem \ref{t:main} allows to obtain a similar result in the regime of
large
sets
$A$.

\bigskip

\begin{theorem}
	Let $A\subseteq \F_p$ be a nonempty set, $\d = |A|/p$.
	Then for any $X\subseteq \F_p^*$,
\begin{equation}\label{cond:E_lx_new}
    |X| \le |A|^{1/2} p^{-1/3}  (\| \FF{A}\|'_\infty)^{4/3}
\end{equation}
one has
\begin{equation}\label{f:E_lx_new}
	\left| \sum_{x \in X} \E^{+} (A, xA) - \frac{|X||A|^4}{p} \right|
        \lesssim
            \log^C (p/|A|) \cdot \d^{8/3} |X|^{1/2} p^{3} \cdot \left( \frac{\| \FF{A}\|'_\infty}{|A|} \right)^{2/3} \,,
\end{equation}
    where $C>0$ is an absolute constant.
	In particular, for any such $X$ the following holds
\begin{equation}\label{f:E_lx_new'}
	\left| \sum_{x \in X} \E^{+} (A, xA) - \frac{|X||A|^4}{p} \right| \lesssim \d^{8/3} |X|^{1/2} p^{3} \,.
\end{equation}
\label{t:E_lx_new}
\end{theorem}
\begin{proof}
	Using formula (\ref{f:energy_Fourier}), we get  as in the proof of Theorem \ref{t:main}
$$
	\sigma:= \sum_{x \in X} \E^{+} (A, xA) = \frac{|X||A|^4}{p} + \frac{1}{p} \sum_{x\in X} \sum_{r\neq 0} |\FF{A}(r)|^2 |\FF{A}(xr)|^2
	=
$$
\begin{equation}\label{tmp:27.07.1}
	=
		\frac{|X||A|^4}{p} + \theta |X| \eps^2 |A|^3 + \frac{1}{p} \sum_{x\in X} \sum_{r\in B} |\FF{A}(r)|^2 |\FF{A}(xr)|^2
		=
			\frac{|X||A|^4}{p} + \theta |X| \eps^2 |A|^3 + \sigma_1 \,,
\end{equation}
where $|\theta|\le 1$, $B=\Spec_\eps (A) \setminus \{0\}$ and
$\eps$ is a parameter,
\begin{equation}\label{f:eps_new}
    \eps^2 \sim (\| \FF{A}\|'_\infty)^{2/3} |X|^{-1/2} \d^{-1} p^{-2/3} \,.
\end{equation}
	Further applying bound (\ref{f:spec_Par}) and our assumption, we obtain
$$
	\d^{} (\| \FF{A}\|'_\infty)^{8/3} p^{1/3}
    = a p^{-2/3}  (\| \FF{A}\|'_\infty)^{8/3}
    \gtrsim |X|^2 \,.
$$
and hence
\begin{equation}\label{tmp:27.07_2}
	|B| \le \frac{1}{\d \eps^2} \lesssim \d^{-1/6} \eps^{-2/3} \sqrt{p} \,.
\end{equation}
Now using the pigeonholing  principle twice, we find two numbers $\Delta_1$, $\Delta_2$ and two sets $W_1,W_2 \subseteq B$ such that
$$
	p\sigma_1 \lesssim  \D_1^2 \D_2^2 \cdot \sum_{x\in X} \sum_r W_1 (r) W_2 (xr)
		=
			 \D_1^2 \D_2^2 \cdot \sum_{x\in X} (W^{-1}_1 \otimes W_2) (x)
$$
and $\D_1 < |\FF{A} (r)|\le 2 \D_1$ for $r\in W_1$, $\D_2 < |\FF{A} (r)|\le 2 \D_2$ for $r\in W_2$.
Applying
Parseval identity (\ref{F_Par}), we get
\begin{equation}\label{tmp:26.07_2'}
 	\D_1^2 |W_1| \le |A| p \,, \quad \quad \D_2^2 |W_2| \le |A| p \,.
\end{equation}
By (\ref{tmp:27.07_2}), we have $|W_1|, |W_2| \le |B| < \d^{-1/6} \eps^{-2/3} \sqrt{p}$.
Using
the first part of Proposition \ref{p:technical_mc} as well as Theorem \ref{t:energy*spec} and formula (\ref{tmp:26.07_2'}), we obtain
$$
	p\sigma_1 \lesssim \D_1^2 \D_2^2 |X|^{1/2} |W_1|^{1/2} |W_2|^{1/2} \d^{-1/3} (\D_1/|A|)^{-2/3} (\D_2/|A|)^{-2/3}
		=
$$
$$
		=
		\D_1^{4/3} \D_2^{4/3} |A|^{4/3} |X|^{1/2} |W_1|^{1/2} |W_2|^{1/2} \d^{-1/3}
			\le
				\D_1^{1/3} \D_2^{1/3} \d^2 p^{10/3} |X|^{1/2} \,.
$$
Using trivial bounds $\D_1, \D_2 \le \| \FF{A}\|'_\infty$ and our choice (\ref{f:eps_new}) of the parameter $\eps$, we obtain, returning to (\ref{tmp:27.07.1}) that
$$
	\left| \sigma - \frac{|X||A|^4}{p} \right| \lesssim |X| \eps^2 \d^3 p^3 + \d^{2} (\| \FF{A}\|'_\infty)^{2/3} |X|^{1/2} p^{7/3}
		\ll
			\d^{2} (\| \FF{A}\|'_\infty)^{2/3} |X|^{1/2} p^{7/3}
    =
$$
$$
    =
        \d^{8/3} |X|^{1/2} p^{3} \cdot \left( \frac{\| \FF{A}\|'_\infty}{|A|} \right)^{2/3} \,.
$$
This completes the proof.
$\hfill\Box$
\end{proof}

\bigskip

For example, if $|A| \le p/2$ then by Parseval identity  (\ref{F_Par}), we get $\| \FF{A}\|'_\infty \gg |A|^{1/2}$
and hence condition (\ref{cond:E_lx_new}) satisfies if $|X|\le |A|^{7/6} p^{-1/3}$.

Notice  that in bound (\ref{f:E_lx_new'}) the term $|X||A|^4/p$ dominates
if $|X| \gtrsim \d^{-8/3}$.
On the other hand, in view of trivial bound
$$
    \left| \sum_{x \in X} \E^{+} (A, xA) - \frac{|X||A|^4}{p} \right| < |A|^2 p
$$
which follows from formula (\ref{tmp:27.07.1}), we see that Theorem \ref{t:E_lx_new} has sense for sets $A$ with small $\| \FF{A}\|'_\infty$ only.

\bigskip

\noindent{I.D.~Shkredov\\
Steklov Mathematical Institute,\\
ul. Gubkina, 8, Moscow, Russia, 119991}
\\
and
\\
IITP RAS,  \\
Bolshoy Karetny per. 19, Moscow, Russia, 127994\\
{\tt ilya.shkredov@gmail.com}

\end{document}